\documentclass[12pt]{amsart}

\usepackage{cite}
\usepackage[hypertexnames=false, colorlinks, linkcolor=blue, allcolors=blue]{hyperref}
\hypersetup{nesting=true,debug=true,naturalnames=true}
\usepackage[margin=1in]{geometry}
\usepackage{graphicx}
\usepackage{tikz}
\usepackage{tikz-cd}
\usepackage{longtable}
\usepackage{multirow}
\usepackage{multicol}
\usepackage[percent]{overpic}
\usepackage{xcolor}
\usepackage{thmtools, thm-restate}
\usepackage{enumitem}

\newtheorem*{thm:det}{Theorem \ref{thm:det}}
\newtheorem*{thm:d-invariants}{Theorem \ref{thm:d-invariants}}
\newtheorem*{conj:2-bridge}{Conjecture \ref{conj:2-bridge}}

\newtheorem{theorem}{Theorem}[section]
\newtheorem{lemma}[theorem]{Lemma}
\newtheorem{conjecture}[theorem]{Conjecture}

\newtheorem{definition}[theorem]{Definition}
\newtheorem{proposition}[theorem]{Proposition}
\newtheorem{corollary}[theorem]{Corollary}

\theoremstyle{definition}
\newtheorem{example}[theorem]{Example}

\newtheorem{remark}[theorem]{Remark}
 
\newcommand{\spinc}{{\textup{Spin}^c}}

\begin{document}

\title{Equivariantly slicing strongly negative amphichiral knots} 

\author{Keegan Boyle and Ahmad Issa}  

\address{Department of Mathematics, University of British Columbia, Canada} 
\email{kboyle@math.ubc.ca}
\address{Department of Mathematics, University of British Columbia, Canada} 
\email{aissa@math.ubc.ca}

\setcounter{section}{0}

\begin{abstract}
We prove obstructions to a strongly negative amphichiral knot bounding an equivariant slice disk in the 4-ball using the determinant, Spin\textsuperscript{c}-structures and Donaldson's theorem. Of the 16 slice strongly negative amphichiral knots with 12 or fewer crossings, our obstructions show that 8 are not equivariantly slice, we exhibit equivariant ribbon diagrams for 5 others, and the remaining 3 are unknown. Finally, we give an obstruction to a knot being strongly negative amphichiral in terms of Heegaard Floer correction terms.
\end{abstract}
\maketitle

\section{Introduction}
A \emph{strongly negative amphichiral} knot $(K,\sigma)$ is a smooth knot $K \subset S^3$ along with a smooth (orientation reversing) involution $\sigma\colon S^3 \to S^3$ such that $\sigma(K) = K$ and $\sigma$ has exactly two fixed points, both of which lie on $K$; see Figure \ref{fig:8_9}. A knot $K \subset S^3$ is \emph{slice} if it bounds a smooth disk (the \emph{slice disk}) properly embedded in $B^4$. The main goal of this paper is to study when there exists an equivariant slice disk for a strongly negative amphichiral knot $(K,\sigma)$. Specifically, we are interested in the following property.

% We draw planar diagrams for strongly negative amphichiral knots in $\mathbb{R}^2$ with a fixed point at the origin and a fixed point at infinity (see Figure \ref{fig:12a458}). The symmetry is then given by a $\pi$-rotation around an axis perpendicular to the diagram followed by reflection across the plane of the diagram.
\iffalse
\begin{figure}
\scalebox{.6}{\includegraphics{Inkscape Diagrams/12a458.pdf}}
\caption{The knot $K = 12a_{458}$. The (unique) strongly negative amphichiral symmetry $\sigma$ is given by $\pi$ rotation around an axis perpendicular to the center of the diagram followed by reflection across the plane of the diagram. The knot $K$ is slice, but we do not know if it has a $\sigma$-invariant slice disk.}
\label{fig:12a458}
\end{figure}
\fi

\begin{definition}
A strongly negative amphichiral knot $(K,\sigma)$ is \emph{equivariantly slice} if there is a smooth slice disk $D$ and a smooth involution $\sigma'\colon B^4 \to B^4$ with $\sigma'(D) = D$ which restricts to $\sigma$ on $\partial B^4 = S^3$. 
\end{definition}

\begin{figure}
\scalebox{.5}{\includegraphics{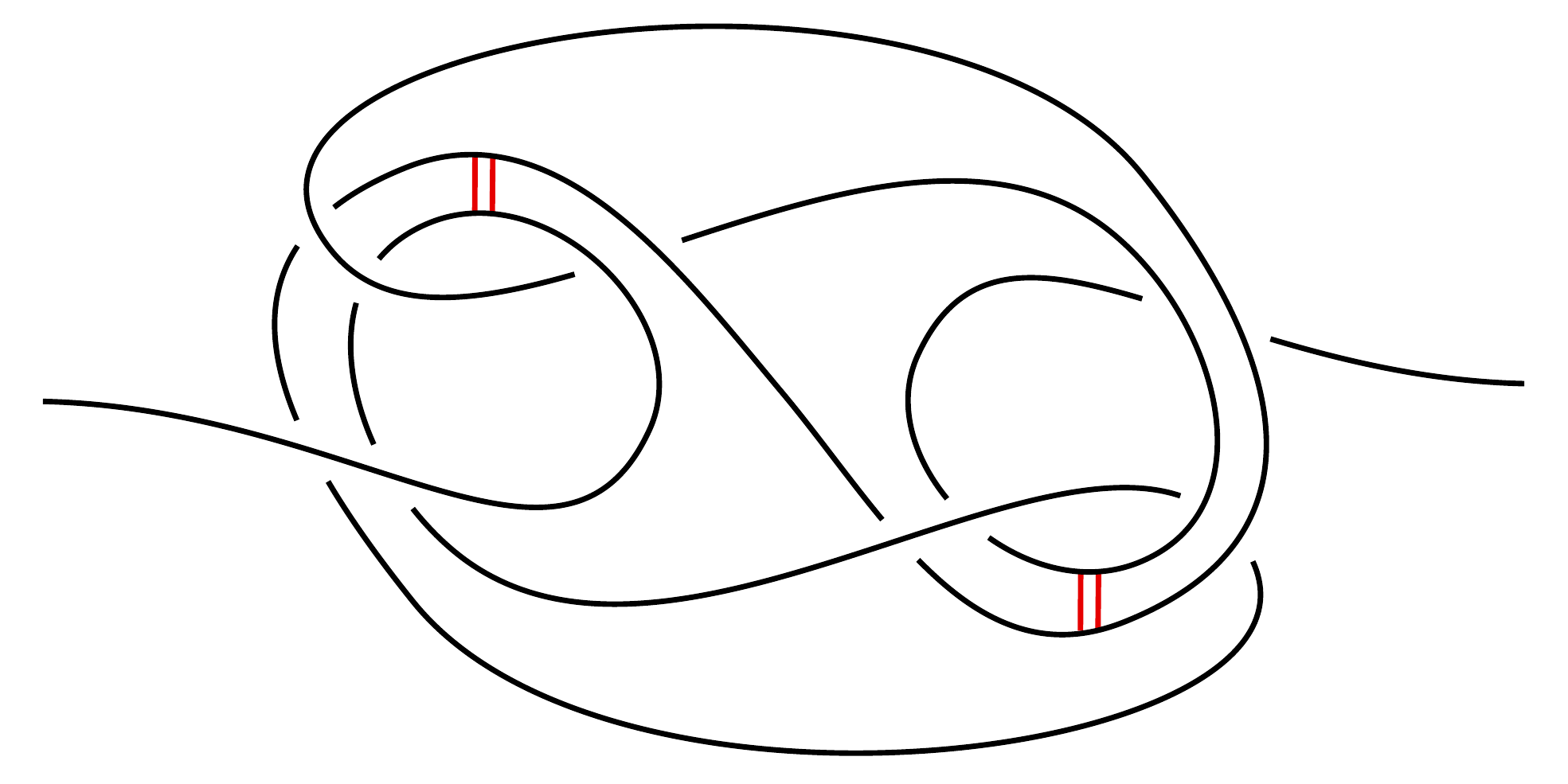}}
\caption{A strongly negative amphichiral diagram for $8_9$. The symmetry is given by $\pi$-rotation around an axis perpendicular to the page followed by a reflection across the plane of the diagram. An equivariant slice disk can be seen by performing the band moves shown in red.}
\label{fig:8_9}
\end{figure}

Figure \ref{fig:8_9} gives an example of a strongly negative amphichiral diagram, that is a knot diagram with the strongly negative amphichiral symmetry given by $\pi$-rotation around an axis perpendicular to the page followed by reflection across the plane of the diagram. Furthermore, the knot in Figure \ref{fig:8_9} is equivariantly slice. The slice disk is given by performing the pair of equivariant band moves shown in red, then equivariantly capping off the resulting 3-component unlink in $B^4$. Among non-trivial prime knots with 12 or fewer crossings there are 16 slice strongly negative amphichiral knots. For five of them, namely $8_9$, $10_{99}$, $12a_{819}$, $12a_{1269}$, and $12n_{462}$,  we found similar equivariant ribbon diagrams; see the table in Section \ref{sec:sliceSNAKStable}.

%Note that strongly negative amphichiral knots cannot bound equivariant surfaces in $S^3$ since the point reflection action at the fixed points does not preserve any surface locally. In particular, we cannot show that a strongly negative amphichiral knot is equivariantly slice by finding an equivariant ribbon disk in $S^3$. Instead, we indicate an invariant pair of bands for which band surgery produces a 3-component unlink (see Figure \ref{fig:8_9}), which always bounds an invariant collection of three disjoint disks in $B^4$. We have found such equivariant ribbon diagrams for $8_9$, $10_{99}$, $12a_{819}$, $12a_{1269}$, and $12n_{462}$; see the table in Section \ref{sec:sliceSNAKStable}.

Strongly negative amphichiral knots, and in particular the equivariant surfaces they bound in the 4-ball, have been studied less than their more popular orientation-preserving cousins: strongly invertible knots, see for example \cite{MR1102258} and \cite{BoyleIssa}, and periodic knots, see for example \cite{MR1605928}, \cite{MR2216254}, \cite{BoyleIssa}, and \cite{GroveJabuka} among others. Many of the obstructions used in the strongly invertible and periodic settings do not adapt to the strongly negative amphichiral case. In fact, even showing that the (non-equivariant) 4-genus for strongly negative amphichiral knots can be arbitrarily large was only recently shown by Miller \cite{Miller}.

%One explanation for this is that they are difficult to study since additive concordance invariants such as the signature vanish, and all strongly negative amphichiral knots are rationally slice \cite[Corollary 1.3]{MR2554510}. 

%In the non-equivariant setting there is more literature. In \cite{MR715764}, Van Buskirk discusses the Alexander polynomials of strongly negative amphichiral knots. In \cite{MR2274531}, Stoimenow shows that the leading coefficient of the Alexander and HOMFLY polynomials of a strongly negative amphichiral knot are perfect squares (up to sign). In \cite{Miller}, Miller uses Casson-Gordon signatures to find strongly negative amphichiral knots with arbitrarily large 4-genus. 

Our first equivariant slice obstruction comes from studying the knot determinant. It was shown by Goeritz \cite{MR1545364} that the determinant of an amphichiral knot is the sum of two squares (see also \cite{FriedlMillerPowell} for a partial generalization and \cite{MR2171693} for the converse). We prove the following strengthening of this determinant condition in the case that $K$ bounds an equivariant slice disk. 

\begin{theorem} \label{thm:det}
 If $K$ is an equivariantly slice strongly negative amphichiral knot, then $\mbox{det}(K)$ is the square of a sum of two squares.
%Let $(K,\sigma)$ be a strongly negative amphichiral knot. If $(K,\sigma)$ is equivariantly slice, then det$(K)$ is the square of a sum of two squares.
\end{theorem}

Theorem \ref{thm:det} shows that the six slice strongly negative amphichiral knots $10_{123}$, $12a_{435}$, $12a_{990}$, $12a_{1019}$, $12a_{1225}$, and $12n_{706}$ are not equivariantly slice.

%Our second obstruction, which applies to alternating knots, comes from studying the $\spinc$-structures on the double cover $W$ of $B^4$ branched over a slice disk. Specifically, the subset of $\spinc$-structures on $Y = \Sigma(S^3,K)$ which extend over $W$ must be $\sigma$-invariant if the slice disk is equivariant; see Proposition \ref{prop:spincinvariance}. To take advantage of this, we take the positive definite checkerboard surface $S$ from an alternating diagram for $K$ and consider the double cover $X$ of $B^4$ branched over $S$ (isotoped to be properly embedded in $B^4$). The manifold $X \cup_Y -W$ obtained by gluing $X$ and $W$ along their common boundary $Y$ is then a positive definite 4-manifold to which we can apply Donaldson's theorem. The result is that we must have a lattice embedding $(H_2(X),Q) \to (\mathbb{Z}^n,Id$), where $Q$ is the intersection form on $X$, and we can enumerate such lattice embeddings combinatorially. 

Our second obstruction, which applies to knots with an alternating strongly negative amphichiral diagram, comes from applying Donaldson's theorem \cite{Donaldson}. Donaldson's theorem can often be used to obstruct the existence of slice disks (see for example \cite{MR2302495}). More recently, it has also been used to obstruct equivariant slice disks for strongly invertible and periodic knots \cite{BoyleIssa}. A key ingredient in that obstruction is the existence of an invariant definite spanning surface for the knot. In contrast, strongly negative amphichiral knots do not bound invariant spanning surfaces in $S^3$. Instead, we use the fact that if $K$ bounds an equivariant slice disk $D$ then the subset $S$ of $\spinc$-structures on the double branched cover $Y = \Sigma(S^3,K)$ that extend over $\Sigma(B^4,D)$ is $\widetilde{\sigma}$-invariant, where $\widetilde{\sigma}$ is a lift of the symmetry $\sigma$ to $Y$; see Proposition \ref{prop:spincinvariance} and the discussion following its proof. Donaldson's theorem can be used to obtain restrictions on $S$. Using the interplay between the pair of checkerboard surfaces exchanged by the symmetry, we carefully keep track of $\spinc$-structures, allowing us to compute the $\widetilde{\sigma}$-action on $\spinc(Y)$. This results in a nice combinatorial description of the $\widetilde{\sigma}$-action on $\spinc(Y)$ in terms of the oriented incidence matrices of the checkerboard graphs for an alternating symmetric diagram. Specifically, we prove the following theorem. 

%Since the $\spinc$-structures on $Y$ can be described in terms of characteristic vectors $\mbox{Char}(H_2(X))$ (see section \ref{sec:spinc} for a definition) in $(H_2(X),Q)$, we get a list of possible subsets of $\spinc$-structures on $Y$ which extend over $W$. 

%One final difficulty is computing the induced action of $\sigma$ on the $\spinc$-structures of $Y = \Sigma(S^3,K)$ so that we can check if the subsets provided by Proposition \ref{prop:spinc_metabolizers} are invariant. To do so, we first define an orientation-preserving symmetry $\widetilde{\sigma}_{swap}$ on $S^4$ which restricts to $\sigma$ on the equatorial $S^3$; see Definition \ref{def:sigmaswap}. We then lift $\widetilde{\sigma}_{swap}$ to a symmetry on the double cover of $S^4$ branched over the union of the checkerboard surfaces $F_+$ and $F_-$ of an equivariant diagram; see Proposition \ref{prop:snalift}. 

%This proposition is useful because the maps $r$,$r_+$,$r_-$, and $\widetilde{\sigma}_{swap}$ are all readily computed combinatorially; see Proposition \ref{prop:dual_embeddings}. 

\begin{theorem} \label{thm:combinatorial_thing}
  Let $(K,\sigma)$ be a knot with an alternating strongly negative amphichiral diagram and let $Y = \Sigma(S^3,K)$. Let $F_{\pm}$ be the positive and negative definite checkerboard surfaces, let $J_{\pm}^*$ be compatible oriented incidence matrices with a row removed\footnote{See Definition \ref{def:incidence}.} for the checkerboard graphs of $F_{\pm}$, and let $A_{\pm} = J_{\pm}^*  (J_{\pm}^*)^{\mathsf{T}} \in M_n(\mathbb{Z})$ be the Goeritz matrices for $F_{\pm}$. Then there is a lift $\widetilde{\sigma}\colon Y \to Y$ for which the map $\widetilde{\sigma}^*\colon\spinc(Y) \to \spinc(Y)$ is determined by 
\[
\widetilde{\sigma}^*[J_+^* v] = [J_-^* v]\ \mbox{ for all } v \in \mathbb{Z}^{2n} \mbox{ with } v \equiv (1,1,\dots, 1)^{\mathsf{T}} \ (\textup{mod 2}),
\]
where $\spinc(Y) \cong \textup{Char}(\mathbb{Z}^n,A_+)/\textup{im}(2A_+)$. Moreover, if $K$ is equivariantly slice then there is a lattice embedding $A: (\mathbb{Z}^n,A_+) \to (\mathbb{Z}^n,\textup{Id})$ such that 
\[
S = \{[u]\in\spinc(Y) \mid u = A^{\mathsf{T}}v\mbox{ for some } v \in \mathbb{Z}^n \mbox{ with } v \equiv (1,1,\dots, 1)^{\mathsf{T}} \ (\textup{mod 2}) \}
\]
is $\widetilde{\sigma}^*$-invariant.

%Moreover, if $K$ is equivariantly slice then there is a lattice embedding $H_2(\Sigma(B^4,F_+)) \to \mathbb{Z}^n$ such that the induced $\spinc(Y)$-metabolizer is invariant under this action. 
\end{theorem}

Using Theorem \ref{thm:combinatorial_thing} we show that $12a_{1105}$ and $12a_{1202}$ are not equivariantly slice (see Section \ref{sec:example}), even though they satisfy the determinant condition in Theorem \ref{thm:det} as det($12a_{1105}) = 17^2 = (4^2 + 1^2)^2$ and det($12a_{1202}) = 13^2 = (3^2 + 2^2)^2$. Of the slice strongly negative amphichiral knots with 12 or fewer crossings, this leaves only $12a_{458}, 12a_{477},$ and $12a_{887}$ for which equivariant sliceness is unknown. See Section \ref{sec:sliceSNAKStable} for a table of equivariant knot diagrams for these knots. 

Our analysis of the $\widetilde{\sigma}$-action on $\spinc(\Sigma(S^3,K))$ also leads us to the following obstruction to strongly negative amphichirality in terms of Heegaard Floer correction terms.

\begin{theorem} \label{thm:d-invariants}
Let $(K,\sigma)$ be a strongly negative amphichiral knot and let $\widetilde{\sigma}$ be a lift of $\sigma$ to $Y := \Sigma(S^3,K)$ (see Proposition \ref{prop:snalift}). Then the orbits of $\spinc(Y)$ under the action of $\widetilde{\sigma}$ take the following form.
\begin{enumerate}
\item There is exactly one orbit $\{\mathfrak{s}_0\}$ of order 1 with $d(Y,\mathfrak{s}_0) = 0$.
\item All other orbits $\{\mathfrak{s},\widetilde{\sigma}(\mathfrak{s}),\widetilde{\sigma}^2(\mathfrak{s}),\widetilde{\sigma}^3(\mathfrak{s})\}$ have order 4 and
\[
d(Y,\widetilde{\sigma}^i(\mathfrak{s})) = (-1)^i \cdot d(Y,\mathfrak{s}) \mbox{ for all } i.
\]
\end{enumerate}
\end{theorem}
For example, the figure eight knot $4_1$ is strongly negative amphichiral and $\Sigma(S^3,4_1) = L(5,2)$, which has correction terms $\{0,2/5,-2/5,2/5,-2/5\}$. We checked that for all $2$-bridge knots with 12 or fewer crossings the $d$-invariants have this structure precisely when the knot is strongly negative amphichiral, leading us to the following conjecture.

\begin{conjecture}\label{conj:2-bridge}
Let $p,q \in \mathbb{N}$ with $p$ odd and $(p,q) = 1$. The following are equivalent:
\begin{enumerate}
\item \label{item:one} The Heegaard Floer correction terms of the lens space $L(p,q)$ can be partitioned into multisets, each of the form $\{r,-r,r,-r\}$ for some $r \in \mathbb{Q}$, and a single set $\{0\}$.
\item \label{item:two} The 2-bridge knot $K(p/q)$ is amphichiral.
\item \label{item:three} There is an orientation-reversing self-diffeomorphism of $L(p,q)$.
\item \label{item:four} $q^2 \equiv -1 \ (\textup{mod } p)$.
\end{enumerate}
\end{conjecture}
We note that \eqref{item:two}, \eqref{item:three}, and \eqref{item:four} are known to be equivalent (see for example \cite[Section 4]{MR2171693} and \cite[Theorem 3]{MR710104}). Theorem \ref{thm:d-invariants} shows that \eqref{item:two} implies \eqref{item:one} since $\Sigma(S^3,K(p/q)) = L(p,q)$ and a 2-bridge knot is amphichiral if and only if it is strongly negative amphichiral. Thus Conjecture \ref{conj:2-bridge} is equivalent to showing that \eqref{item:one} implies any of the other conditions.

\subsection{Acknowledgments}
We thank Liam Watson for his encouragement, support, and interest in this project, and Adam Levine for pointing out a simple proof of Lemma \ref{lemma:tauhomology}.

\section{Lifting the action to the double branched cover}

In this section we show that you can lift the strongly negative amphichiral involution $\sigma$ on $S^3$ to the double branched cover $\Sigma(S^3,K)$. Since we are interested in equivariant slice disks for $K$, we also show that this lift $\widetilde{\sigma}$ can be extended to $\Sigma(B^4,S)$ for any equivariant surface $S \subset B^4$ with $\partial S = K$. Specifically, we have the following proposition which is similar to \cite[Proposition 12]{BoyleIssa}. However in our situation there are no fixed points disjoint from the branch set; the amphichiral involution lifts to an order 4 symmetry on the double branched cover.

\begin{proposition} \label{prop:snalift}
Let $S \subset S^4$ be a closed, connected, smoothly embedded surface and let $\sigma:(S^4,S) \to (S^4,S)$ be a smooth involution with non-empty fixed-point set contained in $S$. Let $p: \Sigma(S^4,S) \to S^4$ be the projection map from the double branched cover and let $\tau: \Sigma(S^4,S) \to \Sigma(S^4,S)$ be the non-trivial deck transformation map. Then there is a lift $\widetilde{\sigma}:\Sigma(S^4,S) \to \Sigma(S^4,S)$ such that the following diagram commutes:

\begin{center}\begin{tikzcd}
  {\Sigma(S^4,S)} & {\Sigma(S^4,S)} \\
  {S^4} & {S^4}
  \arrow["{p}"', from=1-1, to=2-1]
  \arrow["{\widetilde{\sigma}}", from=1-1, to=1-2]
  \arrow["{p}", from=1-2, to=2-2]
  \arrow["{\sigma}"', from=2-1, to=2-2]
\end{tikzcd}.
\end{center}
Furthermore, $\widetilde{\sigma}^2 = \tau$ and there are exactly two such lifts, namely $\widetilde{\sigma}$ and $\widetilde{\sigma}^3$.
\end{proposition}

\begin{proof}
Let $N(S)$ be an equivariant tubular neighborhood of $S$ and $E = S^4 \backslash N(S)$ be the surface exterior. Denote by $\widetilde{E}$ the double cover of $E$ corresponding to the kernel $G$ of $\pi_1(E) \to H_1(E;\mathbb{Z}/2\mathbb{Z})$. We also choose a basepoint $s \in E$ and lifts $\widetilde{s}, \widetilde{t} \in \widetilde{E}$ with $p(\widetilde{s}) = s$ and $p(\widetilde{t}) = \sigma(s)$.

Since $G$ is the unique index 2 subgroup of $\pi_1(E)$, it is a characteristic subgroup. Hence $G$ is also the image of $\pi_1(\sigma \circ p): \pi_1(\widetilde{E}) \to \pi_1(E)$. Then by the covering space lifting property, since Im$(\pi_1(\sigma \circ p)) \subseteq$ Im$(\pi_1(p))$, there is a unique map $\widetilde{\sigma}: (\widetilde{E},\widetilde{s}) \to (\widetilde{E},\widetilde{t})$ such that $p \circ \widetilde{\sigma} = \sigma \circ p$. We next observe that $\widetilde{\sigma}$ preserves the set of $S^1$ fibers on the $S^1$-bundle boundary of $\widetilde{E}$, and by extending this action over each $D^2$ fiber we can (smoothly) extend $\widetilde{\sigma}$ to the tubular neighborhood $p^{-1}(N(S)) \subset \Sigma(S^4,S)$ such that $p \circ \widetilde{\sigma} = \sigma \circ p$.

Finally, $p \circ \widetilde{\sigma} = \sigma \circ p$ implies that $p \circ \widetilde{\sigma}^2 = \sigma^2 \circ p = p$, so that $\widetilde{\sigma}^2$ is either the identity map, or else the non-trivial deck transformation $\tau$ on $\Sigma(B^4,S)$. Note that in either case, $\widetilde{\sigma}^4$ is the identity map. However, $\sigma$ acts by $\pi$-rotation on an equivariant meridian $\alpha$ of a fixed point of $\sigma$. Indeed, if $\sigma$ acted by reflection or identity on $\alpha$, then there would be fixed points disjoint from $S$. In the branched cover we then have that $\widetilde{\sigma}$ acts by $\pi/2$-rotation on $p^{-1}(\alpha)$. Thus $\widetilde{\sigma}$ has order 4 and $\widetilde{\sigma}^2 = \tau$, as desired. Finally, we note that there are exactly two lifts, $\widetilde{\sigma}$ and $\tau \circ \widetilde{\sigma} = \widetilde{\sigma}^3$; one for each choice of $\widetilde{t}$.
\end{proof}

\begin{corollary} \label{cor:snalift}
Let $(K,\sigma)$ be a strongly negative amphichiral knot with double branched cover $\Sigma(S^3,K)$. Let $S \subset B^4$ be a smooth properly embedded surface with boundary $K$ which is invariant under an extension of $\sigma$ to $B^4$ (which we again call $\sigma$). Then there is a lift $\widetilde{\sigma}:\Sigma(B^4,S) \to \Sigma(B^4,S)$ such that $\widetilde{\sigma}^2 = \tau$ (and hence $\widetilde{\sigma}^4 = \mbox{Id}$) and $p \circ \widetilde{\sigma} = \sigma \circ p$. In fact, there are exactly two such lifts namely $\widetilde{\sigma}$ and $\widetilde{\sigma}^3$. %In particular, restricting to the boundary gives a lift $\widetilde{\sigma}:\Sigma(S^3,K) \to \Sigma(S^3,K)$.
\end{corollary}

\begin{proof}
Take the double of $\Sigma(B^4,S)$ to obtain a closed connected surface in $S^4$, then apply Proposition \ref{prop:snalift} and restrict to $\Sigma(B^4,S)$.
\end{proof}

\begin{proposition}\label{prop:snalift3} Let $(K,\sigma)$ be a strongly negative amphichiral knot. Then there exist exactly two lifts of $\sigma$ to $\Sigma(S^3,K)$. Moreover, each such lift $\widetilde{\sigma}$ has $\widetilde{\sigma}^2 = \tau$, where $\tau : \Sigma(S^3,K) \rightarrow \Sigma(S^3,K)$ is the non-trivial deck transformation action, and hence $\widetilde{\sigma}$ has order 4.
  %A similar argument to that in Proposition \ref{prop:snalift} directly shows that there exist exactly two lifts of $\sigma$ to $\Sigma(S^3,K)$ and that they have order 4.
\end{proposition}

\begin{proof} The proof is essentially the same as that of Proposition \ref{prop:snalift}. It can also be obtained by restricting the lifts in Corollary \ref{cor:snalift} to the boundary $\Sigma(S^3, K)$.
\end{proof}

\section{A condition on the determinant}

It is implicit in the work of Goeritz \cite{MR1545364} that the determinant of an amphichiral knot can be written as the sum of two squares (see also \cite{MR2171693} for the converse and \cite{FriedlMillerPowell} for a partial generalization). In this section we reprove this theorem for strongly negative amphichiral knots, and show that the same condition must hold on the square root of the determinant if $K$ is equivariantly slice.

\begin{thm:det}
Let $(K,\sigma)$ be a strongly negative amphichiral knot. Then det$(K)$ is a sum of two squares. Furthermore, if $(K,\sigma)$ is equivariantly slice, then det$(K)$ is the square of a sum of two squares.
\end{thm:det}
Before we give a proof of the theorem, we need a few lemmas.

\begin{lemma} \label{lemma:tauhomology}
Let $A$ be an abelian group, and let $\Sigma(X,Y)$ be the double cover of a manifold $X$ (possibly with boundary) branched over a properly embedded submanifold $Y \subset X$ with non-trivial deck transformation involution $\tau:\Sigma(X,Y) \to \Sigma(X,Y)$. Suppose that $H_n(X;A) = 0$. Then $\tau_*(x) = -x$ for all $x \in H_n(\Sigma(X,Y);A)$.
\end{lemma}

\begin{proof}
The image of the transfer homomorphism $T: H_n(X;A) \to H_n(\Sigma(X,Y);A)$ is 0 since $H_n(X;A) = 0$. For any $x \in H_n(\Sigma(X,Y);A)$, we have that $x + \tau_*(x)$ is in the image of $T$ and hence is $0$. Thus $\tau_*(x) = -x$.

\end{proof}
%\begin{lemma} 
%Let $A$ be an abelian group and let $\tau: \Sigma(S^3,K) \to \Sigma(S^3,K)$ be the deck transformation involution. Then $\tau_*(x) = -x$ for all $x \in H_1(\Sigma(S^3,K);A)$.
%\end{lemma}

%\begin{proof}
%The image of the transfer homomorphism $T:H_1(S^3) \to H_1(\Sigma(S^3,K))$ is 0 since $H_1(S^3) = 0$. For any $x \in H_1(\Sigma(S^3,K))$, we have that $x + \tau_*(x)$ is in the image of $T$ and hence is $0$. Thus $x = -\tau_*(x)$.
%\end{proof}

When $(X,Y) = (S^3,K)$ in Lemma \ref{lemma:tauhomology}, observe that $\tau_*$ fixes only the identity element since $H_1(\Sigma(S^3,K);A)$ has no elements of order 2. 

\begin{lemma}(see \cite[Lemma 3]{MR900252}) \label{lemma:slicedisk}
Let $K$ be slice with slice disk $D \subset B^4$ and let $A$ be a torsion-free abelian group. If the image of $H_1(\Sigma(S^3,K);A)$ in $H_1(\Sigma(B^4,D);A)$ has order $m$, then $|H_1(\Sigma(S^3,K);A)| = m^2$.
\end{lemma}
\begin{proof}
The proof is as in \cite[Lemma 3]{MR900252}, noting that since $A$ is torsion free the universal coefficient theorem does not introduce any unwanted Tor terms.
\end{proof}

\begin{lemma} \label{lemma:equikernel}
Suppose $(K,\sigma)$ has an equivariant slice disk $D$. Then the kernel of the map
\[
i_*: H_1(\Sigma(S^3,K);A) \to H_1(\Sigma(B^4,D);A),
\]
induced by inclusion, is invariant under the induced action of any lift $\widetilde{\sigma} : \Sigma(S^3, K) \rightarrow \Sigma(S^3, K)$ of $\sigma$ on homology.
% $\widetilde{\sigma} : \Sigma(B^4,D) \rightarrow \Sigma(B^4,D)$
%$(\widetilde{\sigma}|_{\Sigma(S^3,K)})_* : H_1(\Sigma(S^3,K);A) \rightarrow H_1(\Sigma(S^3,K);A)$.
\end{lemma}
\begin{proof}
Let $x \in $ ker$(i_*)$ so that $x$ is a boundary in $\Sigma(B^4,D)$. By Corollary \ref{cor:snalift}, there is an extension of the lift $\widetilde{\sigma}$ to $\Sigma(B^4, D)$. Hence $\widetilde{\sigma}_*(x)$ is also a boundary and hence contained in ker$(i_*)$.  
%Let $x \in $ ker$(i_*)$ so that $x$ is a boundary in $\Sigma(B^4,D)$. By Corollary \ref{cor:snalift}, there is a lift $\widetilde{\sigma} : \Sigma(B^4,D) \rightarrow \Sigma(B^4, D)$ extending the lift of $\sigma$ to $\Sigma(S^3, K)$. Then $\widetilde{\sigma}_*(x)$ is also a boundary and hence contained in ker$(i_*)$. 
\end{proof}

\begin{proof}[Proof of Theorem \ref{thm:det}]
By Proposition \ref{prop:snalift3}, $\sigma$ lifts to an order 4 action $\widetilde{\sigma}$ on $\Sigma(S^3,K)$ with $\widetilde{\sigma}^2 = \tau$. In particular, Lemma \ref{lemma:tauhomology} implies that all orbits of $\widetilde{\sigma}_*:H_1(\Sigma(S^3,K); A) \to H_1(\Sigma(S^3,K); A)$ have order 4, except the orbit consisting of the identity element. Taking coefficients $A$ as the $p$-adic integers $\mathbb{Z}_p$ for some prime $p$, we have 
\[
|H_1(\Sigma(S^3,K);\mathbb{Z}_p)| \equiv 1 \ (\mbox{mod } 4).
\]
% If $p \equiv 1$ or 2 (mod 4) then this condition is trivial, but
For $p \equiv 3$ (mod 4) this implies that $|H_1(\Sigma(S^3,K);\mathbb{Z}_p)|$ is an even power of $p$. However by the universal coefficient theorem, $H_1(\Sigma(S^3,K);\mathbb{Z}_p) \cong H_1(\Sigma(S^3,K);\mathbb{Z}) \otimes \mathbb{Z}_p$ and hence the prime decomposition of $|H_1(\Sigma(S^3,K);\mathbb{Z})| = \mbox{det}(K)$ contains an even power of $p$. By the sum of two squares theorem, we then have that det$(K)$ is the sum of two squares.

Now suppose that $(K,\sigma)$ has an equivariant slice disk $D \subset B^4$. By Lemma \ref{lemma:slicedisk} with $p$-adic coefficients, the kernel of $H_1(\Sigma(S^3,K);\mathbb{Z}_p) \to H_1(\Sigma(B^4,D);\mathbb{Z}_p)$ is a square-root order subgroup of $H_1(\Sigma(S^3,K);\mathbb{Z}_p)$, and by Lemma \ref{lemma:equikernel}, this subgroup is invariant under the action of $\widetilde{\sigma}_*$. In particular this subgroup must consist of the identity plus a (finite) collection of order 4 orbits so that 
\[
\sqrt{|H_1(\Sigma(S^3,K);\mathbb{Z}_p)|} \equiv 1 \ (\mbox{mod } 4).
\]
As above, we then have that $\sqrt{\mbox{det}(K)}$ can be written as the sum of two squares.
\end{proof}

\section{An obstruction on spinc structures}
\label{sec:spinc}
In this section we prove Theorem \ref{thm:combinatorial_thing} giving an obstruction to an alternating strongly negative amphichiral knot bounding an equivariant slice disk $D$ in $B^4$. We do so by considering $\spinc$-structures on the double branched cover and applying Donaldson's theorem. This obstruction is based on the following observation.

\begin{proposition} \label{prop:spincinvariance}
Let $\rho:Y \to Y$ be a diffeomorphism of a closed $3$-manifold $Y$. If $\rho$ extends to a diffeomorphism $\rho':X \to X$ of a 4-manifold $X$ with $\partial X = Y$, then
\[
\rho^*(\spinc(X)|_Y) = \spinc(X)|_Y,
\]
where $\rho^*: \spinc(Y) \to \spinc(Y)$ is the induced map on the $\spinc$-structures on the boundary.
\end{proposition}

\begin{proof} Since $\rho'$ is a diffeomorphism, we have that $\rho^*(\spinc(X)|_Y) = (\rho')^*(\spinc(X))|_Y = \spinc(X)|_Y$.
\end{proof}

 In order to use this proposition, take $Y = \Sigma(S^3,K)$, $X = \Sigma(B^4,D)$ and $\rho = \widetilde{\sigma}: \Sigma(B^4,D) \to \Sigma(B^4,D)$ a lift of the strongly negative amphichiral symmetry from Corollary \ref{cor:snalift}. In order to rule out that $\widetilde{\sigma}_*(\spinc(X)|_Y) = \spinc(X)|_Y$ we will need to compute $\widetilde{\sigma}^*:\spinc(Y) \to \spinc(Y)$ and also restrict the possible subsets $\spinc(X)|_Y \subset \spinc(Y)$ using Donaldson's theorem. Proposition \ref{prop:spinc_compute} and Proposition \ref{prop:dual_embeddings} combined allow us to compute $\widetilde{\sigma}^*:\spinc(Y) \to \spinc(Y)$, and Proposition \ref{prop:spinc_metabolizers} gives restrictions on $\spinc(X)|_Y \subset \spinc(Y)$. See Section \ref{sec:example} for an example.

 We recall the following characterization of $\spinc$-structures in terms of characteristic co-vectors which we will use throughout this section. Let $X$ be a simply connected smooth 4-manifold with $\partial X$ either empty or a rational homology sphere. Let $Q$ be the intersection form on $X$ and $\spinc(X)$ be the set of $\spinc$-structures of $X$. Then the first Chern class gives a bijection between the $\spinc$-structures on $X$ and the characteristic co-vectors of $H_2(X)$ (see \cite[Proposition 2.4.16]{MR1707327}). More precisely, we have 
\[
\spinc(X) \cong \mbox{Char}(H_2(X)) := \{u \in H_2(X)^*\ |\ u(x) \equiv Q(x,x)\ (\mbox{mod } 2)\ \forall x \in H_2(X)\}.
\]
If $\partial X \neq \emptyset$ this identification induces a bijection
\[
\spinc(\partial X) \cong \mbox{Char}(H_2(X))/2i(H_2(X)).
\]
where $i:H_2(X) \to H_2(X)^*$ is given by $x \mapsto Q(x,-)$.

The following proposition gives restrictions on the set of $\spinc$-structures on a $3$-manifold which extend over a rational homology $4$-ball which it bounds.

\begin{proposition}
\label{prop:spinc_metabolizers}
Let $X$ be a simply connected positive-definite smooth 4-manifold with boundary a rational homology sphere $Y$. Suppose that $Y$ bounds a rational homology 4-ball $W$. The inclusion map $X \to X \cup_Y -W$ induces an embedding $\iota_*: (H_2(X),Q) \to (\mathbb{Z}^n,\mbox{Id})$ where $Q$ is the intersection form of $X$. Choosing a basis for $H_2(X)$, $\iota_*$ is given by an $n\times n$ matrix $A$ and the $\spinc$-structures on $Y$ which extend over $W$ are those of the form
\[
A^{\mathsf{T}}(v) \pmod{2Q} \in \spinc(Y) = \textup{Char}(H_2(X))/\textup{im}(2Q)
\]
where $v \in \mathbb{Z}^n$ is any vector with all odd entries, and where elements of $\textup{Char}(H_2(X)) \subset \mbox{Hom}(H_2(X), \mathbb{Z})$ are written in the dual basis.
\end{proposition}

\begin{proof}
Let $Z = X \cup_Y -W$, and note that $Z$ is positive definite (see e.g. \cite[Proposition 7]{MR4062777}). Hence by Donaldson's theorem there is an isomorphism of intersection forms $(H_2(Z)/\mbox{Tor},Q_Z) \cong (\mathbb{Z}^n,\mbox{Id})$ where $n = b_2(X)$. We then have a map $\iota_*:(H_2(X),Q) \to (\mathbb{Z}^n,\mbox{Id})$ induced by the inclusion $\iota: X \hookrightarrow Z$. Applying Hom$(-,\mathbb{Z})$ gives the map $\iota^*: H^2(Z)/\mbox{Tor} \to H^2(X)$ which induces a map $\iota^*: \mbox{Char}(H_2(Z)) \to \mbox{Char}(H_2(X))$ on $\spinc$-structures. Recall as well that the restriction $r:\spinc(X) \to \spinc(Y)$ is given by the quotient map
\[
r: \mbox{Char}(H_2(X)) \to \mbox{Char}(H_2(X))/2i(H_2(X))
\]
where $i: H_2(X) \to \textup{Hom}(H_2(X), \mathbb{Z})$ is given by $x \mapsto Q(x,-)$. Hence the restriction map from $\spinc(Z) \to \spinc(Y)$ is given by $r \circ \iota^*$. We then claim that the image of $r \circ \iota^*$ is precisely the $\spinc$-structures on $Y$ which extend over $W$. Indeed $r$ is surjective, so all $\spinc$-structures on $Y$ extend over $X$, and hence a $\spinc$-structure on $Y$ extends over $W$ if and only if it extends over all of $Z$.

Combinatorially, we can compute this restriction as follows. Choose a basis for $H_2(X)$, and the dual basis for $\textup{Hom}(H_2(X), \mathbb{Z})$. Then $\iota_*$ is given by a matrix $A$, and $\iota^*$ is given by $A^{\mathsf{T}}$. The characteristic covectors of $H_2(Z)$ are given by vectors $v$ in $\mathbb{Z}^n$ with all odd entries. Then the image of $\iota^*$ consists of elements of all vectors of the form
\[
  % A^{\mathsf{T}}\begin{bmatrix} 1 \\ \vdots \\ 1 \end{bmatrix} + \mbox{im}(2A^{\mathsf{T}}) \subset \textup{Char}(H_2(X)) = \spinc(X),
  A^{\mathsf{T}} v \in \textup{Char}(H_2(X)) = \spinc(X),
\]
written in the dual basis for $\textup{Hom}(H_2(X), \mathbb{Z}) \supset \textup{Char}(H_2(X))$. The image of $r \circ \iota^*$ then consists of these vectors modulo the column space of $2Q$.
\end{proof} 

We now turn to computing $\widetilde{\sigma}^*:\spinc(\Sigma(S^3,K)) \to \spinc(\Sigma(S^3,K))$. To do so, begin with a strongly negative amphichiral alternating diagram for $K$, and let $F_+$ and $F_-$ be the pair of checkerboard surfaces with $F_+$ and $F_-$ positive and negative definite respectively. Note that $F_+$ and $F_-$ are exchanged by the strongly negative amphichiral symmetry. 

\begin{definition} \label{def:sigmaswap}
Take $S^4$ as the unit sphere in $\mathbb{R}^5$. Define $\sigma_{swap}: S^4 \to S^4$ as the involution
\[
(x_1,x_2,x_3,x_4,x_5) \mapsto (x_1,-x_2,-x_3,-x_4,-x_5).
\]
\end{definition}
On the equatorial $S^3 = \{(x_1,x_2,x_3,x_4,0): x_1^2 + x_2^2 + x_3^2 + x_4^2 = 1\}$, $\sigma_{swap}$ restricts to the (unique\footnote{Livesay \cite{MR155323} proved that up to conjugation there is a unique involution on $S^3$ with exactly two fixed points.}) amphichiral symmetry $\sigma$ with two fixed points $(\pm1,0,0,0,0)$. Finally, note that $\sigma_{swap}$ is orientation-preserving and exchanges the two hemispheres of $S^4$.

%\begin{lemma} \label{lemma:sigmaswap}
%Let $\sigma:S^3 \to S^3$ be the amphichiral symmetry as above. Then there is an orientation-preserving involution $\sigma_{swap}:S^4 \to S^4$ which restricts to $\sigma$ on the equatorial $S^3$ and exchanges the two hemispheres of $S^4$. Furthermore the fixed-point set of $\sigma_{swap}$ is two points, both of which lie in $S^3$.
%\end{lemma}
%\begin{proof}
%Take $S^4$ as the unit sphere in $\mathbb{R}^5$. Then define $\sigma_{swap}$ as 
%\[
%(x_1,x_2,x_3,x_4,x_5) \mapsto (x_1,-x_2,-x_3,-x_4,-x_5)
%\]
%where the equatorial $S^3 = \{(x_1,x_2,x_3,x_4,0): x_1^2 + x_2^2 + x_3^2 + x_4^2 = 1\}$. The fixed points are $(\pm1,0,0,0,0) \subset S^3$.
%\end{proof}

With respect to this involution $\sigma_{swap}$, we can push $F_+$ and $F_-$ equivariantly into distinct hemispheres of $S^4$. By Proposition \ref{prop:snalift} there are two lifts $\widetilde{\sigma}_{swap}$ and $\widetilde{\sigma}_{swap}'$ of $\sigma_{swap}$ to an order 4 symmetry of $\Sigma(S^4,F_+ \cup F_-)$. We have that $\widetilde{\sigma}_{swap} = \widetilde{\sigma}_{swap}' \circ \tau$ where $\tau$ is the non-trivial deck transformation involution $\tau: \Sigma(S^4,F_+ \cup F_-) \to \Sigma(S^4,F_+ \cup F_-)$. Using Lemma \ref{lemma:tauhomology} this implies that
\[
-(\widetilde{\sigma}_{swap})_* = (\widetilde{\sigma}_{swap}')_*: H_2(\Sigma(S^4,F_+ \cup F_-)) \to H_2(\Sigma(S^4,F_+ \cup F_-)).
\]
This immediately implies the following proposition. 

\begin{proposition} \label{prop:GLhomologymap}
Let $\widetilde{\sigma}_{swap}, \widetilde{\sigma}_{swap}'$ be the two lifts of $\sigma_{swap}$ to $\Sigma(S^4,F_+ \cup F_-)$. These lifts induce maps $H_2(\Sigma(B^4,F_+)) \to H_2(-\Sigma(B^4,F_-))$ which are equal to $\pm\sigma_*:H_1(F_+) \to H_1(F_-)$ under the identification of $H_2(\Sigma(B^4,F_{\pm}))$ with $H_1(F_{\pm})$ from \cite[Theorem 3]{MR500905}.
\end{proposition}

%The idea is then to take a spin$^c$-structure $\overline{\mathfrak{s}}$ on $\Sigma(S^4,F_+ \cup F_-)$, and restrict it to a spin$^c$-structure on $\Sigma(S^3,K)$ in two different ways. On one hand, we can restrict directly to a spin$^c$-structure $\mathfrak{s}$. On the other hand, we can first restrict to $\spinc(\Sigma(B^4,F_+))$, apply $(\widetilde{\sigma}_{swap})^*$, then restrict to $\mathfrak{s}' \in \spinc(\Sigma(S^3,K))$.We then have that $\widetilde{\sigma}^*(\mathfrak{s}) = \mathfrak{s}'$.
We now use $\widetilde{\sigma}_{swap}$ to help us understand the action of $\widetilde{\sigma}$ on $\spinc$-structures.
\begin{proposition}
  \label{prop:spinc_compute}
Let $(K,\sigma)$ be a strongly negative amphichiral knot and fix a lift $\widetilde{\sigma}:\Sigma(S^3,K) \to \Sigma(S^3,K)$ (see Proposition \ref{prop:snalift3}). The induced action $\widetilde{\sigma}^*: \spinc(\Sigma(S^3,K)) \to \spinc(\Sigma(S^3,K))$ can be computed as follows. Let $\mathfrak{s} \in \spinc(\Sigma(S^3,K))$, let $r,r_-,$ and $r_+$ be the obvious restriction maps in the following non-commutative diagram, and let $\overline{\mathfrak{s}} \in \spinc(\Sigma(S^4, F_+ \cup F_-))$ such that $r \circ r_+(\overline{\mathfrak{s}}) = \mathfrak{s}$. Then $\widetilde{\sigma}^*(\mathfrak{s}) = r \circ (\widetilde{\sigma}_{swap})^* \circ r_-(\overline{\mathfrak{s}})$, where the lift $\widetilde{\sigma}_{swap}$ is chosen to agree with $\widetilde{\sigma}$ on $\Sigma(S^3,K)$.

\begin{center}\begin{tikzcd}
{\spinc(\Sigma(S^3,K))} & {\spinc(\Sigma(B^4,F_+))} &{\spinc(\Sigma(S^4,F_+ \cup F_-))} \\
&{\spinc(-\Sigma(B^4,F_-))}&
\arrow["r"', from=1-2, to=1-1]
\arrow["r_+"', from=1-3, to=1-2]
\arrow["r_-", from=1-3, to=2-2]
\arrow["({\widetilde{\sigma}_{swap})^*}", from=2-2, to=1-2]
\end{tikzcd}
\end{center}

\end{proposition}
\begin{proof}
By construction, $\widetilde{\sigma}_{swap}|_{\Sigma(S^3,K)} = \widetilde{\sigma}^*$. Hence the map 
\[
(\widetilde{\sigma}_{swap})^*: \spinc(\Sigma(S^4,F_+ \cup F_-)) \to \spinc(\Sigma(S^4,F_+ \cup F_-))
\] 
restricts to $\widetilde{\sigma}^*: \spinc(\Sigma(S^3,K)) \to \spinc(\Sigma(S^3,K))$. We then compute
\begin{align*}
\widetilde{\sigma}^*(\mathfrak{s}) &= (\widetilde{\sigma}_{swap})^*(\mathfrak{s})\\
&= (\widetilde{\sigma}_{swap})^* \circ r \circ r_+(\overline{\mathfrak{s}}) \\
&= r \circ r_+ \circ (\widetilde{\sigma}_{swap})^*(\overline{\mathfrak{s}}) \\
&= r \circ (\widetilde{\sigma}_{swap})^* \circ r_-(\overline{\mathfrak{s}}),
\end{align*}
where $r_+ \circ (\widetilde{\sigma}_{swap})^* = (\widetilde{\sigma}_{swap})^* \circ r_-$ since $\widetilde{\sigma}_{swap}$ exchanges $\Sigma(B^4,F_+)$ and $\Sigma(B^4,F_-)$ in $\Sigma(S^4,F_+ \cup F_-)$.
\end{proof}

Following \cite{MR500905}, we associate to each vertex $v_i$ of the checkerboard graph $\mathcal{G}(F_+)$ an element of $H_2(\Sigma(B^4,F_+))$, which we again refer to as $v_i$ as follows. The vertex $v_i$ corresponds to a planar region of the knot diagram. Let $\gamma_i$ be the simple loop in $F_+$ running once counterclockwise around this region. Applying the isomorphism $H_1(F_+) \cong H_2(\Sigma(B^4,F_+))$ from \cite[Theorem 3]{MR500905} we get an element $v_i \in H_2(\Sigma(B^4,F_+))$. We call the $\{v_i\}$ the \emph{vertex generating set} of $H_2(\Sigma(B^4,F_+))$.

% To Do: add definition of an oriented incidence matrix?
\begin{definition}\label{def:incidence}
  Let $D_K$ be a strongly negative amphichiral alternating knot diagram, let $F_{\pm}$ be the positive and negative definite checkerboard surfaces and let $\mathcal{G}(F_{\pm})$ be the corresponding checkerboard graphs, embedded as dual planar graphs. The graphs $\mathcal{G}(F_{\pm})$ are \emph{compatibly oriented} if their edges are oriented such that intersecting dual edges satisfy the right hand rule as in the left of Figure \ref{fig:arcorientation}.

  Suppose $\mathcal{G}(F_{\pm})$ are compatibly oriented and order the vertices of each of $\mathcal{G}(F_{\pm})$ so that the strongly negative amphichiral symmetry respects the orderings, and enumerate the edges of each graph so that intersecting edges have the same index; see Figure \ref{fig:12a1105graph} for an example. We call the oriented incidence matrices $J_{\pm}$ for $\mathcal{G}(F_{\pm})$ \emph{compatible}. We use the notation $J_+^*$ (resp. $J_-^*$) to denote the matrix $J_+$ (resp. $J_-$) with the last row removed. Recall that in an oriented incidence matrix $A$, 
\[A_{i,j} = \begin{cases} 
      1 & \mbox{if the } j \mbox{th edge begins at the } i \mbox{th vertex,} \\
      -1 & \mbox{if the } j \mbox{th edge terminates at the } i \mbox{th vertex, and} \\
      0 & \mbox{otherwise.} 
   \end{cases}
\]
\end{definition}

The following proposition can be used to combinatorially compute the maps $r_+$ and $r_-$ from Proposition \ref{prop:spinc_compute} in terms of oriented incidence matrices; see Remark \ref{rmk:r+-}.

\begin{proposition}
\label{prop:dual_embeddings}
Let $D$ be an alternating knot diagram with positive and negative definite checkerboard surfaces $F_+$ and $F_-$ respectively, and let $\mathcal{G}(F_{\pm})$ be compatibly oriented checkerboard graphs, see Definition \ref{def:incidence}. Then there is an orthonormal basis $\{e_i\}$ of $H_2(\Sigma(S^4,F_+ \cup F_-))$ in bijection with the crossings of $D$ for which the maps $H_2(\pm\Sigma(B^4,F_{\pm})) \to H_2(\Sigma(S^4,F_+ \cup F_-))$ induced by inclusion are given by the transposes $(J_{\pm})^{\mathsf{T}}$ of the oriented incidence matrices of $\mathcal{G}(F_{\pm})$ with respect to the vertex generating sets for $H_2(\pm\Sigma(B^4,F_{\pm}))$.
\end{proposition}
\begin{proof}
Following \cite[proof of Theorem 3]{MR500905}, $\Sigma(B^4, F_+)$ (and similarly $\Sigma(B^4,F_-)$) can be constructed as follows. Let $D_1$ denote the manifold obtained by cutting open $B^4$ along the trace of an isotopy which pushes $\mbox{int} (F_+)$ into $\mbox{int} (B^4)$. The manifold $D_1$ is homeomorphic to $B^4$ and the part exposed by the cut is given by a tubular neighborhood $N_+$ of $F_+$ in $S^3 \cong \partial D_1$. Let $D_2$ be another copy of $D_1$, and let $\iota \colon N_+ \rightarrow N_+$ be the involution given by reflecting each fiber. Then
\[
\Sigma(B^4, F_+) = (D_1 \cup -D_2) / (x \in N_+ \subset D_1 \sim \iota(x) \in N_+ \subset D_2).
\]
There is an isomorphism $\phi \colon (H_1(F_+), Q_{F_+}) \rightarrow (H_2(\Sigma(B^4,F_+)), Q_+)$, where $Q_{F_+}$ is the Gordon-Litherland form and $Q_+$ is the intersection form, which is given as follows. Let $a$ be a 1-cycle in $F_+$, then
\[
\phi([a]) = [\mbox{(cone on }a\mbox{ in }D_1\mbox{)} - \mbox{(cone on }a\mbox{ in }D_2\mbox{)}].
\]
%Taking $a$ as a counterclockwise simple loop around a region of $F_-$, $\phi[a]$ is an element of the vertex basis for $F_+$. 
The surfaces $F_+$ and $F_-$ in $S^3$ intersect in a collection of $k$ arcs $\alpha_1, \dots, \alpha_k$, one for each crossing of $D$. The $I$-subbundle of $N_+$ over $\alpha_i$ is a disk $D^2_+(\alpha_i) \subset D_1$ with boundary $\widetilde{\alpha}_i$, the preimage of $\alpha_i$ in $\Sigma(S^3,K)$. (The disk $D^2_+(\alpha_i)$ is also the trace of $\alpha_i$ under the isotopy pushing int$(F_+)$ into int$(B^4)$.) Note that $D^2_+(\alpha_i)$ is properly embedded in $\Sigma(B^4,F_+)$. Similarly, there is a disk $D^2_-(\alpha_i)$ properly embedded in $\Sigma(B^4,F_-)$, and gluing these disks along $\widetilde{\alpha}_i$ gives a sphere $e_i$ in $\Sigma(S^4,F_+ \cup F_-)$. 

Note that $e_1, \dots, e_k$ are in correspondence with the edges of $\mathcal{G}(F_+)$ (and $\mathcal{G}(F_-)$). Furthermore, the orientation on an edge $E_i$ in $\mathcal{G}(F_+)$ induces an orientation on the corresponding $e_i$ as follows. First, orient the arc $\alpha_i$ going into the page of the knot diagram (away from the reader). Next, push the interior of $\alpha_i$ into the region corresponding to the terminal vertex of $E_i$ and then out of the page of the diagram (toward the reader) so that it is disjoint from $F_+ \cup F_-$. Call the resulting arc $\alpha_i'$; see Figure \ref{fig:arcorientation}. Recall that $\Sigma(B^4,F_+) = D_1 \cup - D_2$ as an oriented manifold. Then the orientation of $\alpha_i' \subset D_1$ determines an orientation on the union of $\alpha_i' \subset D_1$ with $-\alpha_i' \subset -D_2$, which is locally isotopic within $\Sigma(S^3,K)$ to $\widetilde{\alpha}_i$. This orientation on $\widetilde{\alpha}_i$ then determines an orientation on $D^2_+(\alpha_i)$ as its oriented boundary, and this orientation on $D^2_+(\alpha_i)$ extends to an orientation on $e_i = D^2_+(\alpha_i) \cup D^2_-(\alpha_i)$.

\begin{figure}
\begin{overpic}[width=250pt, grid = false]{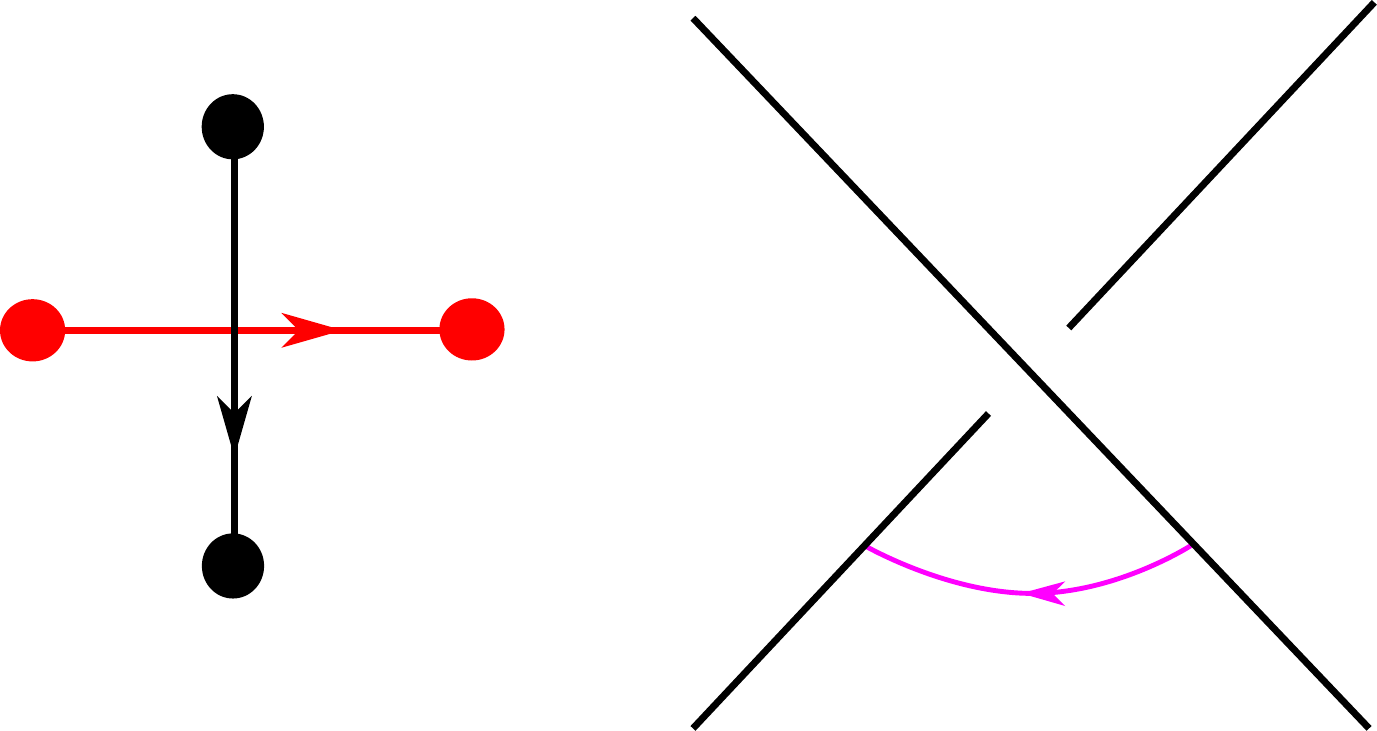}
\put (70, 4) {$\alpha_i'$}
\end{overpic}
\caption{An oriented edge of $\mathcal{G}(F_+)$ in black intersecting an edge of $\mathcal{G}(F_-)$ in red (left). The orientation on the red edge is induced by the right hand rule. On the right is the oriented arc $\alpha_i'$ induced from the oriented edge of $\mathcal{G}(F_+)$ in black.}
\label{fig:arcorientation}
\end{figure}

We now show that $\{e_1,\dots,e_k\}$ is an orthonormal basis for $H_2(\Sigma(S^4,F_+ \cup F_-))$. Note that $b_2(\Sigma(S^4,F_+ \cup F_-)) = b_2(\Sigma(B^4,F_+)) + b_2(\Sigma(B^4,F_-))$ since $\Sigma(S^3,K)$ is a rational homology sphere. However, $b_2(\Sigma(B^4,F_{\pm})) = n_{\pm} - 1$, where $n_{\pm}$ is the number of vertices of $\mathcal{G}(F_{\pm})$. From the Euler characteristic of the sphere of the knot diagram, we get $2 = n_+ - k + n_-$ since $\mathcal{G}(F_+)$ and $\mathcal{G}(F_-)$ are dual graphs. Hence $b_2(\Sigma(S^4,F_+ \cup F_-)) = k$. Thus it suffices to show that $e_1,\dots,e_k$ are orthonormal. Observe that $e_i$ and $e_j$ are disjoint for $i \neq j$ so it is enough to show that $e_i \cdot e_i = 1$. Consider the arcs $(\alpha_i)_{\pm}$ shown in Figure \ref{fig:e_i} where $(\alpha_i)_{\pm} \subset F_{\pm}$, and $(\alpha_i)_+$ intersects $(\alpha_i)_-$ at a single point. Observe that the preimages $(\widetilde{\alpha_i})_{\pm} \subset \Sigma(B^4,F_{\pm})$ of $(\alpha_i)_{\pm}$ bound disks $D^2_{\pm}(\alpha_i)'$ parallel to $D^2_{\pm}(\alpha_i)$ in $\Sigma(B^4,F_{\pm})$. There is an isotopy in $S^3$ between $(\alpha_i)_+$ and $(\alpha_i)_-$ intersecting $\alpha_i$ in a single point which induces an isotopy between $(\widetilde{\alpha_i})_+$ and $(\widetilde{\alpha_i})_-$. Gluing $D^2_+(\alpha_i)'$ to $D^2_-(\alpha_i)'$ along the (image of the) isotopy in $\Sigma(S^3,K)$ defines a push-off of $e_i$ which has a single positive transverse intersection with $e_i$. 

Recall that an element $v_i \in H_2(\Sigma(B^4,F_+))$ of the vertex generating set is represented by a sphere which intersects $N_+ \subset \Sigma(B^4,F_+)$ in a loop $\gamma_i \subset F_+$. By construction, $e_j \cap \Sigma(B^4,F_+)$ is the disk $D^2_+(\alpha_j)$ contained in $N_+ \subset \Sigma(B^4,F_+)$. Hence $v_i \cdot e_j$ can be computed locally in $N_+$. Diagrammatically (see Figure \ref{fig:vertexintersection}), we draw $\partial D_1 = S^3$ and think of $N_+$ as a neighborhood of $F_+ \subset S^3$. Specifically, $v_i \cdot e_j = 0$ if the edge corresponding to $e_j$ and $v_i$ are not incident, $v_i \cdot e_j = 1$ if the edge corresponding to $e_j$ begins at $v_i$, and $v_i \cdot e_j = -1$ if the edge corresponding to $e_j$ terminates at $v_i$. A similar argument applies to the vertex generating set of $H_2(-\Sigma(B^4,F_-))$.

\end{proof}

\begin{figure}
\begin{overpic}[width=150pt, grid = false]{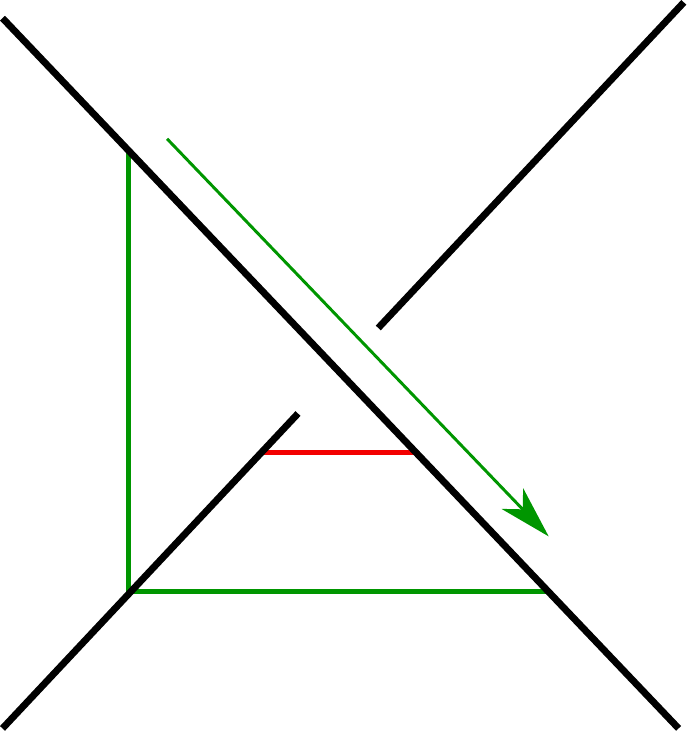}
\put (-8,40) {$(\alpha_i)_+$}
\put (40,6) {$(\alpha_i)_-$}
\put (40,30) {$\alpha_i$}
\end{overpic}
\caption{The arcs $(\alpha_i)_+$ and $(\alpha_i)_-$ are contained in the horizontal and vertical checkerboard surfaces respectively. The green arrow indicates an isotopy between them in $S^3$. Lifting this to $\Sigma(S^3,K)$, we see that the self pairing of the sphere $e_i$ is 1.}
\label{fig:e_i}
\end{figure}

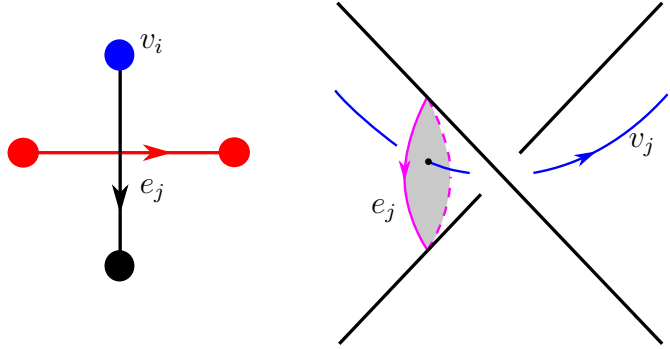
\begin{figure}
\begin{overpic}[width=250pt, grid = false]{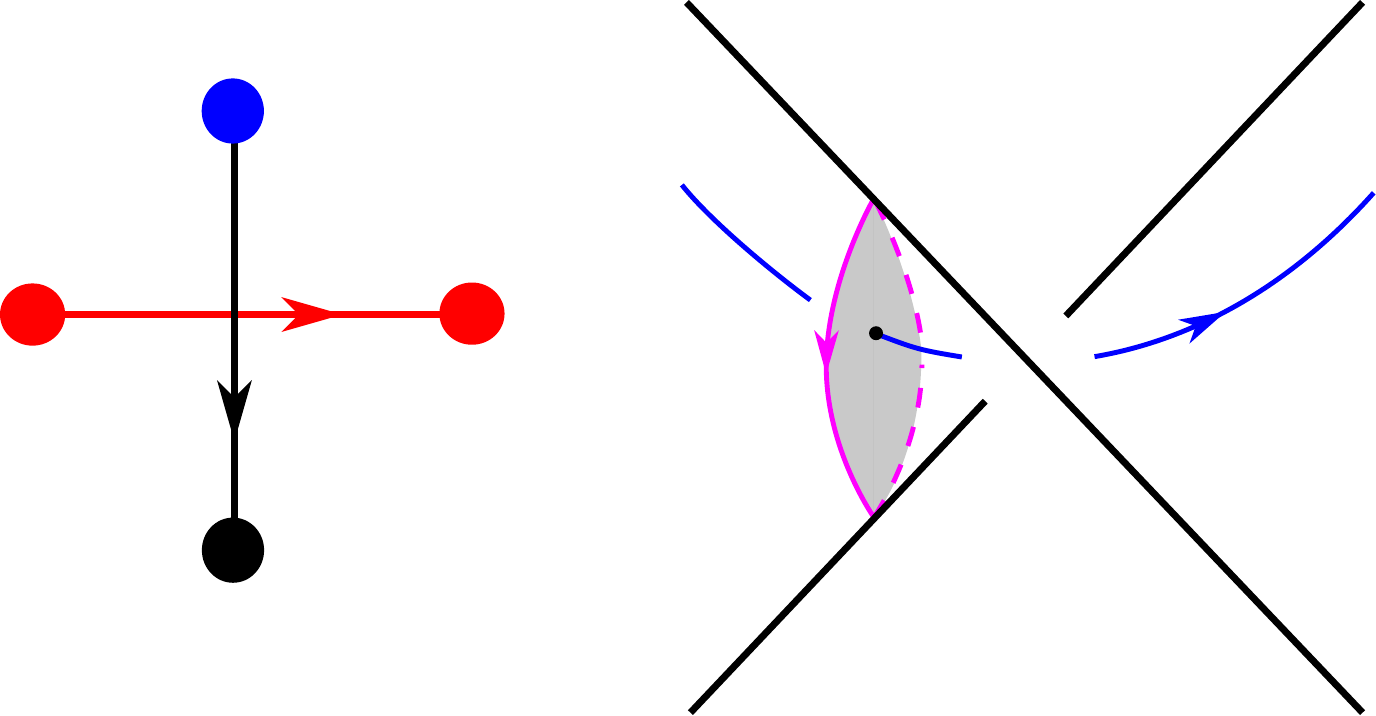}
\put (20,45) {$v_i$}
\put (20,23) {$e_j$}
\put (94, 30) {$v_j$}
\put (55,20) {$e_j$}
\end{overpic}
\caption{If $v_i \in \mathcal{G}(F_+)$ is the starting endpoint of an edge corresponding to $e_j$, then $e_j \cdot v_i = 1$. The magenta loop is the boundary of the gray disk $e_j \cap N_+$, and is oriented so that the arc coming out of the page is isotopic (keeping the endpoints on $K$) to $\alpha_j'$ (see Figure \ref{fig:arcorientation}) in the complement of $F_+ \cup F_-$.}
\label{fig:vertexintersection}
\end{figure}

\begin{remark} \label{rmk:r+-}
Note that Proposition \ref{prop:dual_embeddings} combinatorially determines the maps 
\[
r_{\pm}: \spinc(\Sigma(S^4,F_+ \cup F_-)) \to \spinc(\pm \Sigma(B^4,F_{\pm}))
\]
from Proposition \ref{prop:spinc_compute}. Specifically the maps $r_{\pm}$ are given by taking the duals of 
\[
H_2(\pm\Sigma(B^4,F_{\pm})) \to H_2(\Sigma(S^4,F_+ \cup F_-)),
\]
then restricting to characteristic vectors.
\end{remark}

We conclude the section with a proof of Theorem \ref{thm:combinatorial_thing} from the introduction.
\begin{proof}[Proof of Theorem \ref{thm:combinatorial_thing}]
  Let $Y = \Sigma(S^3, K)$ and $X_{\pm} = \Sigma(B^4, F_{\pm})$. We identify each of $H_2(X_{\pm})$ with the $\mathbb{Z}$-span of $\mbox{Vert}(\mathcal{G}(F_{\pm})) \backslash \{v_{\pm}\}$, where $\{v_+, v_-\}$ is the pair of $\sigma$-invariant vertices removed when defining $J_{\pm}^*$. Using the dual basis for $H_2(X_{\pm})^*$ we identify
  \[
  \spinc(X_{\pm}) \cong \mbox{Char}(\mathbb{Z}^n, A_{\pm})\mbox{ and } \spinc(Y) \cong \mbox{Char}(\mathbb{Z}^n, A_{+})/\textup{im}(2A_{+}).
  \]
  With respect to these choices of dual bases, we may choose a lift $\widetilde{\sigma}$ of $\sigma$ to $Y$ so that $\widetilde{\sigma}_{swap}^* : H_2(-X_{-})^* \rightarrow H_2(X_{+})^*$ is the identity matrix by Proposition \ref{prop:GLhomologymap}; this determines the map on $\spinc$-structures. Since $Y$ is a rational homology sphere, $b_2(\Sigma(B^4, F_+ \cup F_-)) = b_2(\Sigma(B^4, F_+)) + b_2(\Sigma(B^4,F_-)) = n + n$. Using the orthonormal basis for $H_2(\Sigma(B^4, F_+ \cup F_-)) \cong \mathbb{Z}^{2n}$ from Proposition \ref{prop:dual_embeddings}, we may identify
  \[
    \spinc(\Sigma(B^4, F_+ \cup F_-)) \cong \{v \in \mathbb{Z}^{2n} : v \equiv (1, 1, \ldots, 1)^{\mathsf{T}} (\mbox{mod } 2)\}.
  \]
  By Proposition \ref{prop:dual_embeddings} (see also Remark \ref{rmk:r+-}), the maps $r_{\pm}$ in Proposition \ref{prop:spinc_compute} are given by $J_{\pm}^*$. Proposition \ref{prop:spinc_compute} then shows that the map $\widetilde{\sigma}^*:\spinc(Y) \to \spinc(Y)$ is determined by 
\[
  \widetilde{\sigma}^*[J_+^* v] = [J_-^* v]\ \mbox{ for all } v \in \mathbb{Z}^{2n} \mbox{ with } v \equiv (1,1,\dots, 1)^{\mathsf{T}} \ (\textup{mod 2}).
\]
Finally, let $D$ be an equivariant slice disk for $K$. By Proposition \ref{prop:spinc_metabolizers}, the set of $\spinc$-structures of $Y$ which extend over $\Sigma(B^4, D)$ is given by
\[
S = \{[u]\in\spinc(Y) : u = A^{\mathsf{T}}v\mbox{ for some } v \in \mathbb{Z}^n \mbox{ with } v \equiv (1,1,\dots, 1)^{\mathsf{T}} \ (\textup{mod 2}) \},
\]
and by Corollary \ref{cor:snalift} there is a lift $\Sigma(B^4, D) \rightarrow \Sigma(B^4, D)$ which restricts to the lift $\widetilde{\sigma}$ on $Y$. Hence by Proposition \ref{prop:spincinvariance}, $S$ is $\widetilde{\sigma}^*$-invariant.
\end{proof}

\section{An alternating slice strongly negative amphichiral example} \label{sec:example}
In this section we give an example of a strongly negative amphichiral knot which Theorem \ref{thm:combinatorial_thing} shows is not equivariantly slice.

\begin{figure}
\begin{overpic}[width=350pt, grid=false]{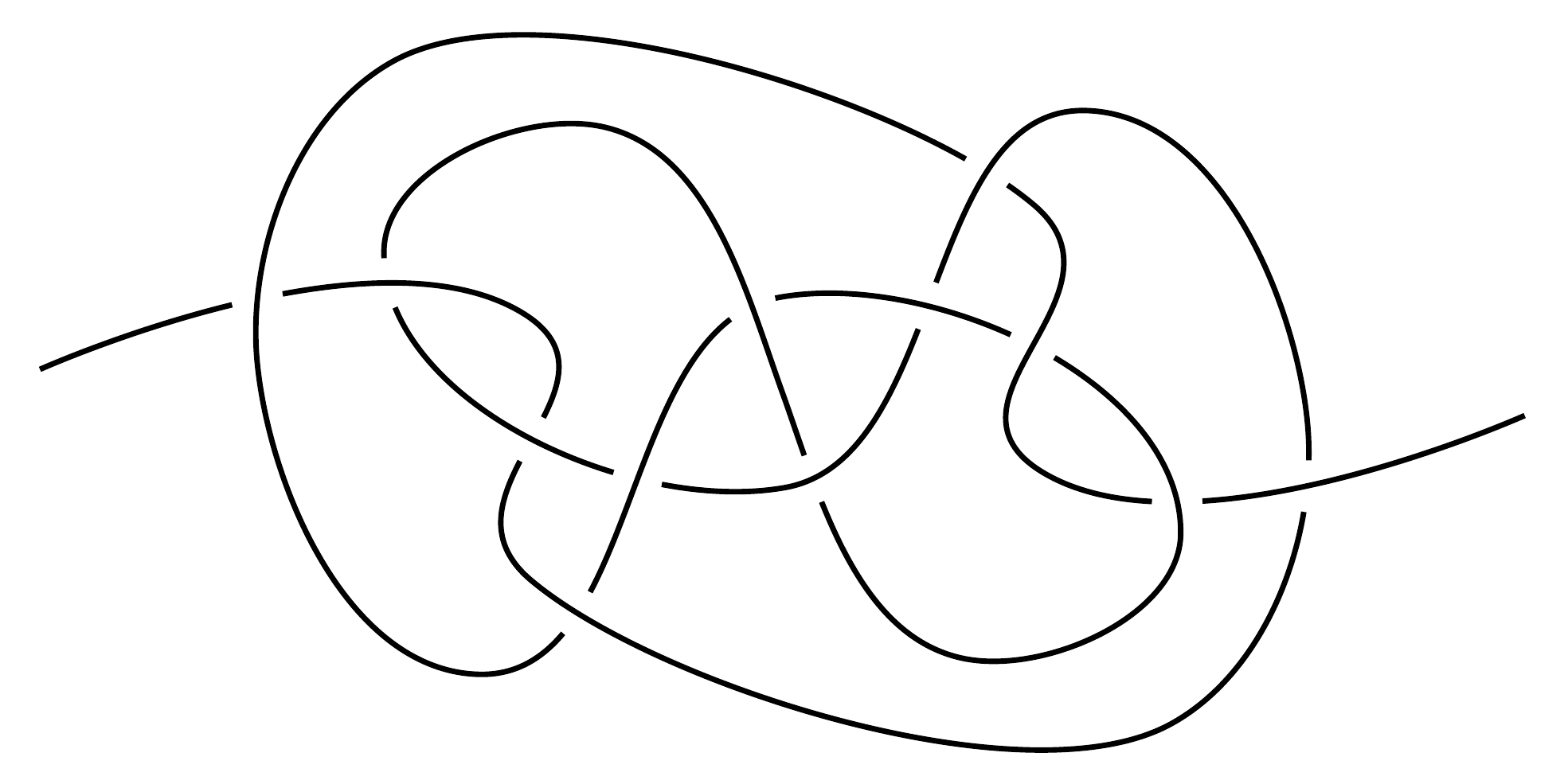}
  \end{overpic}
  \caption{A strongly negative amphichiral symmetry on $12a_{1105}$. The symmetry is $\pi$-rotation within the plane of the diagram followed by a reflection across the plane of the diagram.}
  \label{fig:12a1105}
\end{figure}

\begin{example}
Consider the slice knot $K = 12a_{1105}$ along with the strongly negative amphichiral alternating diagram shown in Figure \ref{fig:12a1105}. Theorem \ref{thm:combinatorial_thing} obstructs $K$ from being equivariantly slice. Note that Theorem \ref{thm:det} does not provide an obstruction since det$(K) = 17^2$. Let $F_+$ (resp. $F_-$) be the positive (resp. negative) definite checkerboard surface for the knot diagram in Figure \ref{fig:12a1105}. In Figure \ref{fig:12a1105graph} we draw corresponding compatibly oriented checkerboard graphs $\mathcal{G}(F_{\pm})$. The edges in each graph are enumerated by the crossings $e_i$ shown in Figure \ref{fig:12a1105graph}. Using $u_7$ and $v_7$ for the last row of the oriented incidence matrices $J_{\pm}$ (which we remove), we have
\[
J_+^*=\begin{bmatrix}
1 & 0 & 0 & 0 & -1 & 1 & 0 & 0 & 0 & 0 & 0 & 0 \\
-1 & 1 & 0 & 0 & 0 & 0 & 0 & 1 & 0 & 0 & 0 & 0 \\
0 & -1 & 1 & 0 & 0 & 0 & 0 & 0 & 1 & 0 & 0 & 1 \\
0 & 0 & -1 & 1 & 0 & 0 & 0 & 0 & 0 & 1 & 0 & -1 \\
0 & 0 & 0 & -1 & 1 & 0 & 0 & 0 & 0 & 0 & 1 & 0 \\
0 & 0 & 0 & 0 & 0 & -1 & 1 & 0 & 0 & 0 & 0 & 0 \\
\end{bmatrix},
\]
and
\[
J_-^*=\begin{bmatrix}
0 & 0 & 0 & 0 & 0 & 0 & 0 & 0 & -1 & 1 & 0 & 1 \\
0 & 1 & 0 & 0 & 0 & 0 & 0 & -1 & 1 & 0 & 0 & 0 \\
1 & 0 & 0 & 0 & 0 & -1 & -1 & 1 & 0 & 0 & 0 & 0 \\
0 & 0 & 0 & 0 & 1 & 1 & 1 & 0 & 0 & 0 & -1 & 0 \\
0 & 0 & 0 & 1 & 0 & 0 & 0 & 0 & 0 & -1 & 1 & 0 \\
0 & 0 & 1 & 0 & 0 & 0 & 0 & 0 & 0 & 0 & 0 & -1 \\
\end{bmatrix}.
\]
From these we can compute the Goeritz matrix for $F_+$:
\[
A_+ = J_+^*(J_+^*)^{\mathsf{T}} = \begin{bmatrix}
3 & -1 & 0 & 0 & -1 & -1 \\
-1 & 3 & -1 & 0 & 0 & 0 \\
0 & -1 & 4 & -2 & 0 & 0 \\
0 & 0 & -2 & 4 & -1 & 0 \\
-1 & 0 & 0 & -1 & 3 & 0 \\
-1 & 0 & 0 & 0 & 0 & 2 
\end{bmatrix}.
\]
We now combinatorially enumerate all possible lattice embeddings $A: (\mathbb{Z}^6,A_+) \to (\mathbb{Z}^6,$ Id), up to automorphisms of $\mathbb{Z}^6$. That is we enumerate integer matrices $A$ satisfying $A^{\mathsf{T}}A = A_+$, up to permutations and sign changes of the rows of $A$. We find two possibilities for $A$ which we denote $A_1$ and $A_2$; their transposes are
\[
A_1^{\mathsf{T}} = \begin{bmatrix}
-1&1&1&0&0&0\\
0&-1&0&1&1&0\\
1&0&1&0&-1&-1\\
-1&-1&0&-1&0&1\\
0&0&-1&1&-1&0\\
1&0&0&0&0&1
\end{bmatrix} \mbox{ and } A_2^{\mathsf{T}} =
\begin{bmatrix}
-1&1&1&0&0&0\\0&-1&0&1&1&0\\
-1&0&-1&-1&0&1\\
1&1&0&0&1&-1\\
0&0&-1&1&-1&0\\
1&0&0&0&0&1
\end{bmatrix}.
\]

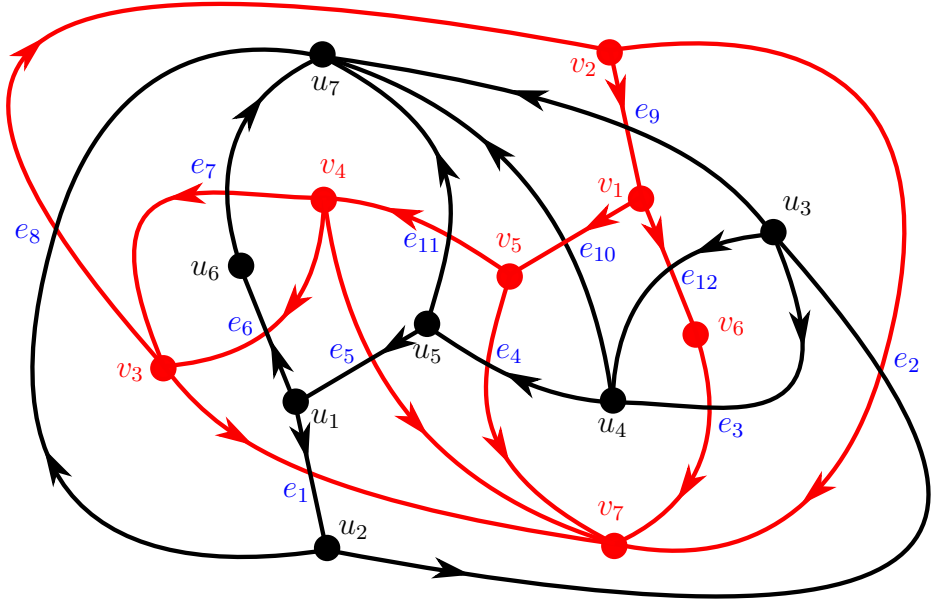
\begin{figure}
  \begin{overpic}[width=350pt, grid=false]{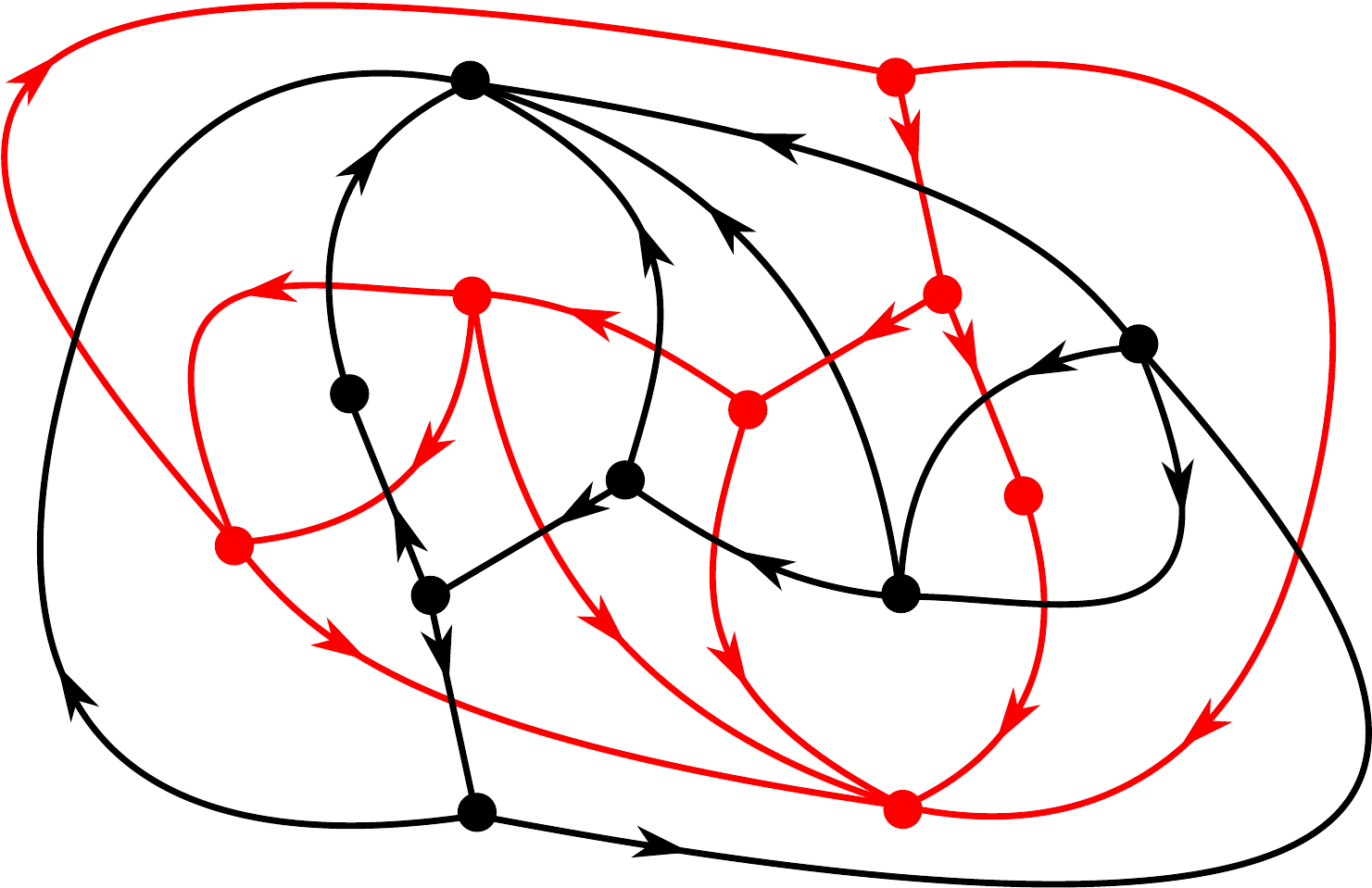}
    \put (33, 55) {$u_7$}
    \put (20, 35) {$u_6$}
    \put (33, 19) {$u_1$}
    \put (44, 26) {$u_5$}
    \put (36, 7) {$u_2$}
    \put (64, 18) {$u_4$}
    \put (84, 42) {$u_3$}

    \put (64, 9) {\textcolor{red}{$v_7$}}
    \put (77, 29) {\textcolor{red}{$v_6$}}
    \put (64, 44) {\textcolor{red}{$v_1$}}
    \put (53, 38) {\textcolor{red}{$v_5$}}
    \put (61, 57) {\textcolor{red}{$v_2$}}
    \put (34, 46) {\textcolor{red}{$v_4$}}
    \put (12, 24) {\textcolor{red}{$v_3$}}

    \put (30, 11) {\textcolor{blue}{$e_1$}}
    \put (96, 25) {\textcolor{blue}{$e_2$}}
    \put (77, 18) {\textcolor{blue}{$e_3$}}
    \put (53, 26) {\textcolor{blue}{$e_4$}}
    \put (35, 26) {\textcolor{blue}{$e_5$}}
    \put (24, 29) {\textcolor{blue}{$e_6$}}
    \put (20, 46) {\textcolor{blue}{$e_7$}}
    \put (1, 39) {\textcolor{blue}{$e_8$}}
    \put (68, 52) {\textcolor{blue}{$e_9$}}
    \put (62, 37) {\textcolor{blue}{$e_{10}$}}
    \put (43, 38) {\textcolor{blue}{$e_{11}$}}
    \put (73, 34) {\textcolor{blue}{$e_{12}$}}

  \end{overpic}
  \caption{The pair of checkerboard graphs of the alternating diagram for $12a_{1105}$ in Figure \ref{fig:12a1105}. They are exchanged by the strongly negative amphichiral symmetry. $\mathcal{G}(F_+)$ is black and $\mathcal{G}(F_-)$ is red. The $\{e_i\}$ correspond to crossings in the knot diagram.}
  \label{fig:12a1105graph}
\end{figure}
Neither matrix satisfies the $\widetilde{\sigma}^*$-invariance condition in Theorem \ref{thm:combinatorial_thing}. We will show this for the matrix $A_1$; the computation for $A_2$ is similar. For $A_1$, we compute that the set
\[
S = \{[u]\in\spinc(Y) : u = A_1^{\mathsf{T}}v\mbox{ for some } v \in \mathbb{Z}^n \mbox{ with } v \equiv (1,1,\dots, 1)^{\mathsf{T}} \ (\textup{mod 2}) \}
\]
consists of the 17 classes represented by the following vectors:

\begin{align*} \bigg\{ \left[
\begin{array}{c}
 1 \\
 1 \\
 -2 \\
 -2 \\
 1 \\
 -2 \\
\end{array}
\right],&\left[
\begin{array}{c}
 3 \\
 -3 \\
 2 \\
 0 \\
 -1 \\
 -2 \\
\end{array}
\right],\left[
\begin{array}{c}
 3 \\
 1 \\
 0 \\
 -2 \\
 -1 \\
 -2 \\
\end{array}
\right],\left[
\begin{array}{c}
 -1 \\
 1 \\
 -4 \\
 4 \\
 -1 \\
 0 \\
\end{array}
\right],\left[
\begin{array}{c}
 1 \\
 -3 \\
 4 \\
 -2 \\
 -1 \\
 0 \\
\end{array}
\right],\left[
\begin{array}{c}
 1 \\
 1 \\
 -2 \\
 4 \\
 -3 \\
 0 \\
\end{array}
\right],\left[
\begin{array}{c}
 3 \\
 -1 \\
 -2 \\
 2 \\
 -3 \\
 0 \\
\end{array}
\right],\left[
\begin{array}{c}
 -1 \\
 -3 \\
 6 \\
 -4 \\
 -1 \\
 2 \\
\end{array}
\right] , \left[
\begin{array}{c}
 -1 \\
 1 \\
 0 \\
 2 \\
 -3 \\
 2 \\
\end{array}
\right], \\ &\left[
\begin{array}{c}
 1 \\
 -1 \\
 0 \\
 -2 \\
 3 \\
 -2 \\
\end{array}
\right],\left[
\begin{array}{c}
 1 \\
 3 \\
 -6 \\
 4 \\
 1 \\
 -2 \\
\end{array}
\right],\left[
\begin{array}{c}
 -3 \\
 1 \\
 2 \\
 -2 \\
 3 \\
 0 \\
\end{array}
\right],\left[
\begin{array}{c}
 -1 \\
 -1 \\
 2 \\
 -4 \\
 3 \\
 0 \\
\end{array}
\right],\left[
\begin{array}{c}
 -1 \\
 3 \\
 -4 \\
 2 \\
 1 \\
 0 \\
\end{array}
\right],\left[
\begin{array}{c}
 1 \\
 -1 \\
 4 \\
 -4 \\
 1 \\
 0 \\
\end{array}
\right],\left[
\begin{array}{c}
 -3 \\
 -1 \\
 0 \\
 2 \\
 1 \\
 2 \\
\end{array}
\right],\left[
\begin{array}{c}
 -3 \\
 3 \\
 -2 \\
 0 \\
 1 \\
 2 \\
\end{array}
\right] \bigg\}.
\end{align*}
We will show that this collection $S$ of $\spinc$-structures on $\Sigma(S^3,K)$ is not $\widetilde{\sigma}^*$-invariant. Specifically, we will show that the $\spinc$-structure represented by the second vector $\mathfrak{s} = (3,-3,2,0,-1,-2)^{\mathsf{T}}$ is mapped by $\widetilde{\sigma}^*$ to a $\spinc$-structure not contained in $S$. 

Consider the vector 
\[\widetilde{\mathfrak{s}} = (7,3,3,3,1,-3,-5,1,1,1,1,1)^{\mathsf{T}} \in \mathbb{Z}^{12}.
\]
Multiplying, we see that $J_+^*(\widetilde{\mathfrak{s}}) = \mathfrak{s}$ and $J_-^*(\widetilde{\mathfrak{s}}) = (1,3,16,-8,3,2)^{\mathsf{T}}$. A straightforward linear algebra computation shows that $(1,3,16,-8,3,2)^{\mathsf{T}}$ is not equivalent (mod $2A_+$) to any of the 17 vectors in $S$. Hence $\widetilde{\sigma}^*[J_+^*(\widetilde{\mathfrak{s}})] = [J_-^*(\widetilde{\mathfrak{s}})]$ is not in $S$. Along with a similar computation for $A_2$, this implies that $K$ is not equivariantly slice by Theorem \ref{thm:combinatorial_thing}.
\end{example}
\section{Heegaard Floer correction terms}
In this section we give a necessary condition on the Heegaard Floer correction terms $d(\Sigma(S^3,K),\mathfrak{s})$, also known as $d$-invariants, for a knot to be strongly negative amphichiral. In the case of periodic knots, a similar type of condition was proved by Jabuka and Naik in \cite{MR3519536}. As in the case of periodic knots, we first need invariance of the $d$-invariants.

\begin{lemma} \label{lemma:d-invariants}
Let $Y$ be a rational homology 3-sphere with $\mathfrak{s} \in \spinc(Y)$ and $\sigma:Y \to Y$ an orientation reversing diffeomorphism. Then 
\[
d(Y,\sigma^*(\mathfrak{s})) = -d(Y,\mathfrak{s}).
\]
\end{lemma}
\begin{proof}
This follows directly from the diffeomorphism invariance of Heegaard Floer homology.
\end{proof}
This implies the following theorem.
\begin{thm:d-invariants}
Let $(K,\sigma)$ be a strongly negative amphichiral knot and let $\widetilde{\sigma}$ be a lift of $\sigma$ to $\Sigma(S^3,K)$ (see Proposition \ref{prop:snalift}). Then the orbits of the $d$-invariants of $\Sigma(S^3,K)$ under the action of $\widetilde{\sigma}$ are of the following form.
\begin{itemize}
\item There is a exactly one orbit $\{\mathfrak{s}_0\}$ of order 1. Moreover, $d(\Sigma(S^3,K),\mathfrak{s}_0) = 0$.
\item All other orbits $\{\mathfrak{s},\widetilde{\sigma}(\mathfrak{s}),\widetilde{\sigma}^2(\mathfrak{s}),\widetilde{\sigma}^3(\mathfrak{s})\}$ have order 4. Moreover, $d(\Sigma(S^3,K),\widetilde{\sigma}^i(\mathfrak{s})) = (-1)^i \cdot r$ for some $r \in \mathbb{Q}$.
\end{itemize}
\end{thm:d-invariants}

\begin{proof}
Let $F \subset S^3$ be an arbitrary spanning surface for $K$ and let $X = \Sigma(B^4,F)$. Then we can choose $\{u_i\}$ so that $H_2(X) = \mathbb{Z}\langle u_1,\dots,u_n \rangle$. Moreover, $\spinc(\Sigma(S^3,K)) = \mbox{Char}(\mathbb{Z} \langle u_i \rangle)/\textup{im}(2Q)$ where $Q$ is the intersection form matrix of $H_2(X)$ in the basis $\{u_i\}$ and $\mbox{Char}(\mathbb{Z} \langle u_i \rangle) \subseteq \mathbb{Z}\langle u_i^* \rangle \cong H_2(X)^*$. Since $\widetilde{\sigma}^2$ is the deck transformation action $\tau$ by Proposition \ref{prop:snalift} and $\tau_*:H_2(X) \to H_2(X)$ is $-$Id by Lemma \ref{lemma:tauhomology}, then the induced action $(\widetilde{\sigma}^*)^2:H_2(X)^* \to H_2(X)^*$ is also $-$Id. Since $\widetilde{\sigma}$ has order 4, the $\widetilde{\sigma}^*$-orbits of the $\spinc$-structures will have order 1, 2, or 4. 

For $j \in \{1,2\}$, let $v_j \in \mbox{Char}(\mathbb{Z} \langle u_i \rangle)$ represent a $\spinc$-structure $[v_j]$ on $\Sigma(S^3,K)$ with $(\widetilde{\sigma}^*)^2([v_j]) = [v_j]$ so that the orbit of $[v_j]$ has order 1 or 2. Then $[v_j] = -[v_j]$ so that $2v_j = 2Qw_j$ and hence $v_j = Qw_j$ for some $w_j \in \mathbb{Z}^n$. Subtracting these we have that $Q(w_1-w_2) = v_1 - v_2 \equiv 0$ (mod 2) since $v_j$ is a characteristic covector. Then since det($Q) = $ det$(K)$ which is odd, $Q$ is invertible over $\mathbb{Z}/2$. Multiplying by $Q^{-1}$ (mod 2), we get that $w_1 - w_2 \equiv 0$ (mod 2) as well. In particular, $v_1 - v_2$ is in the image of $2Q$ so that $[v_1] = [v_2]$. Hence there is a unique $\spinc$-structure with an orbit of order 1 or 2. Finally, there are an odd number of $\spinc$-structures, so there must be an orbit of order 1 and all other orbits have order 4. The statements about the $d$-invariants now follow from Lemma \ref{lemma:d-invariants}.
\end{proof}

\begin{example}
The $d$-invariants of $\Sigma(S^3,6_1)$, appropriately oriented, are 
\[
\dfrac{-4}{9},\dfrac{-4}{9},0,0,0,\dfrac{2}{9},\dfrac{2}{9}, \dfrac{8}{9},\dfrac{8}{9}.
\]
Since these are not of the form required by Theorem \ref{thm:d-invariants}, $6_1$ is not strongly negative amphichiral. We compare this to the strongly negative amphichiral knot $6_3$, for which $\Sigma(S^3,6_3)$ has $d$-invariants
\[
0,\ \dfrac{8}{13},\dfrac{-8}{13},\dfrac{8}{13},\dfrac{-8}{13},\ \dfrac{6}{13},\dfrac{-6}{13},\dfrac{6}{13},\dfrac{-6}{13},\ \dfrac{2}{13},\dfrac{-2}{13},\dfrac{2}{13},\dfrac{-2}{13}.
\]
\end{example}

\iffalse
In fact, we checked that for 2-bridge knots with 12 or fewer crossings the $d$-invariants satisfy Theorem \ref{thm:d-invariants} precisely when the knot is strongly negative amphichiral, leading to the following conjecture.
\begin{conj:2-bridge}
If $K$ is a 2-bridge knot then $K$ is strongly negative amphichiral if and only if the condition in Theorem \ref{thm:d-invariants} is satisfied.
\end{conj:2-bridge}
\begin{remark}
Since 2-bridge knots are alternating and prime, they are either hyperbolic or torus knots \cite{MR721450}. Torus knots are chiral, and a hyperbolic knot is negative amphichiral if and only if it is strongly negative amphichiral. Thus the above conjecture would characterize all negative amphichiral 2-bridge knots in terms of $d$-invariants.
\end{remark}
\fi

\section{A table of slice strongly negative amphichiral prime knots with 12 or fewer crossings}\label{sec:sliceSNAKStable}
We conclude with a table of all slice strongly negative amphichiral prime knots with 12 or fewer crossings. These are categorized as follows:
\begin{enumerate}[align=parleft, labelsep=0cm,]
\item[(Rib):] \,\,\,\, Knots for which we have found an equivariant ribbon diagram. We indicate this with a pair of equivariant bands (in red) which reduce the knot to a 3-component unlink.
\item[(Det):] \,\,\,\, Knots for which Theorem \ref{thm:det} obstructs an equivariant slice disk.
\item[($\spinc$):] \,\,\,\, Knots for which the obstruction from Theorem \ref{thm:det} fails but Theorem \ref{thm:combinatorial_thing} obstructs an equivariant slice disk.
\item[(Unk):] \,\,\,\, Knots for which we were unable to find or obstruct an equivariant slice disk.
\end{enumerate}
We also include the knot determinant and whether the knot is equivariantly slice.
\begin{longtable}{c|c|c|c|c}
Name & Diagram & Eq. slice & Category & Det \\*\hline \endhead \hline \hline
 & \multirow{9}{*}{\scalebox{.4}{\includegraphics{Inkscape Diagrams/8_9Ribbon.pdf}}}&&&\\*
 $8_9$&&Yes&(Rib)&$5^2$\\*
 &&&&\\*
 &&&&\\*
 &&&&\\*
 &&&&\\*
 &&&&\\*
 &&&&\\*
 &&&&\\\hline\hline

 & \multirow{9}{*}{\scalebox{.4}{\includegraphics{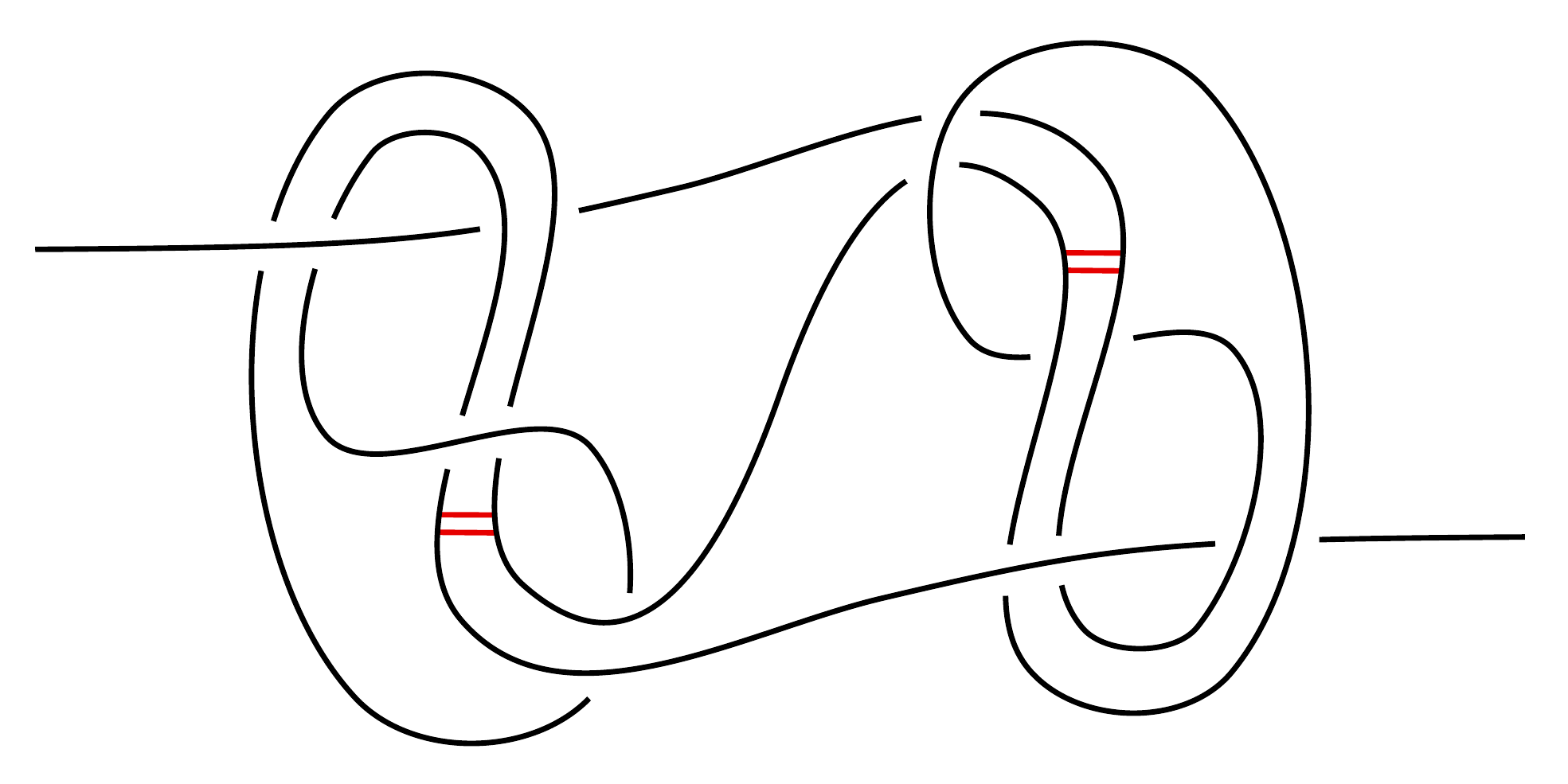}}}&&&\\*
 $10_{99}$&&Yes&(Rib)&$9^2$\\*
 &&&\\*
 &&&\\*
 &&&\\*
 &&&\\*
 &&&\\*
 &&&\\*
 &&&\\\hline\hline

 & \multirow{9}{*}{\scalebox{.4}{\includegraphics{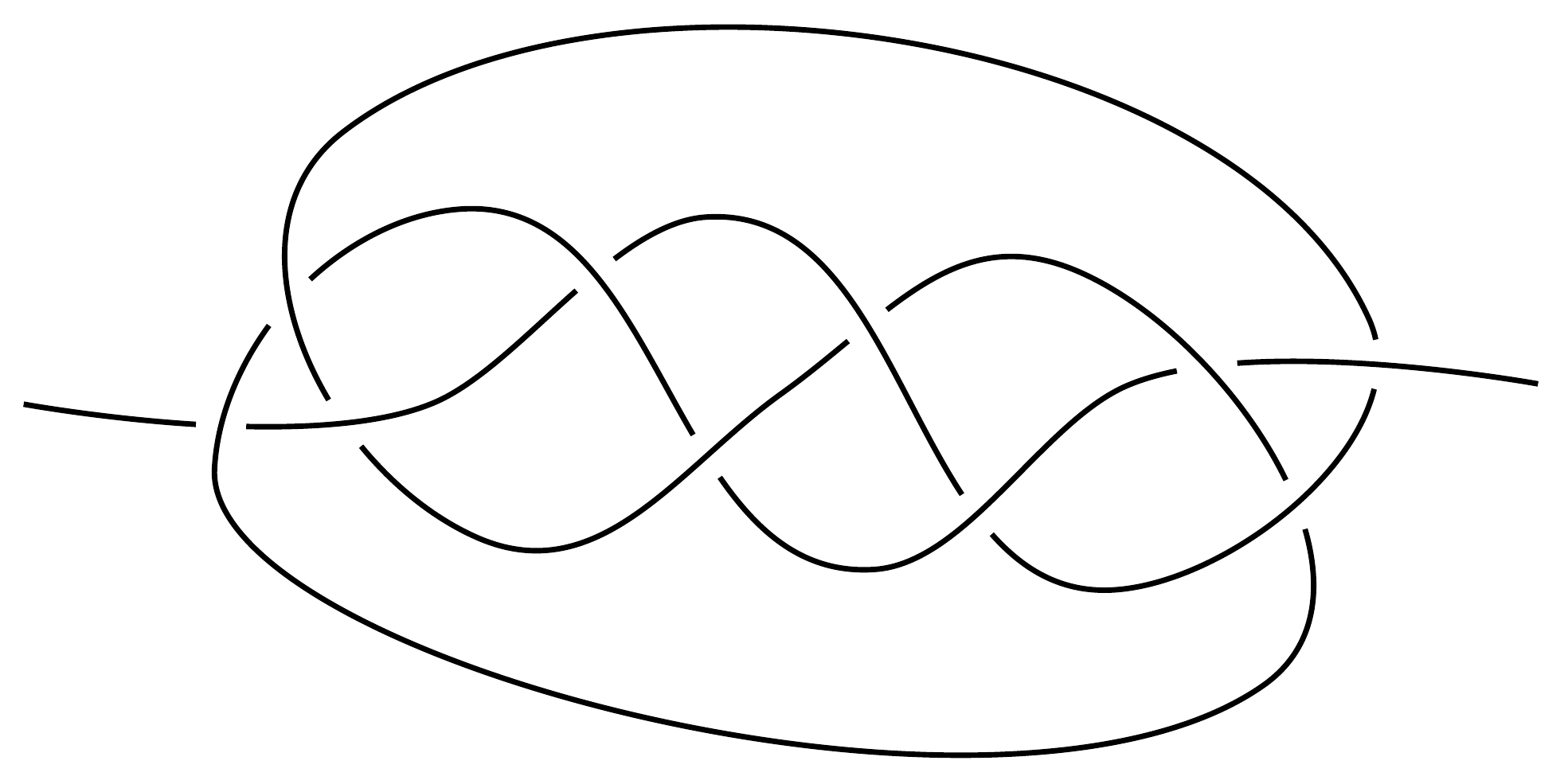}}}&&&\\*
 $10_{123}$&&No&(Det)&$11^2$\\*
 &&&\\*
 &&&\\*
 &&&\\*
 &&&\\*
 &&&\\*
 &&&\\*
 &&&\\\hline\hline

 & \multirow{9}{*}{\scalebox{.4}{\includegraphics{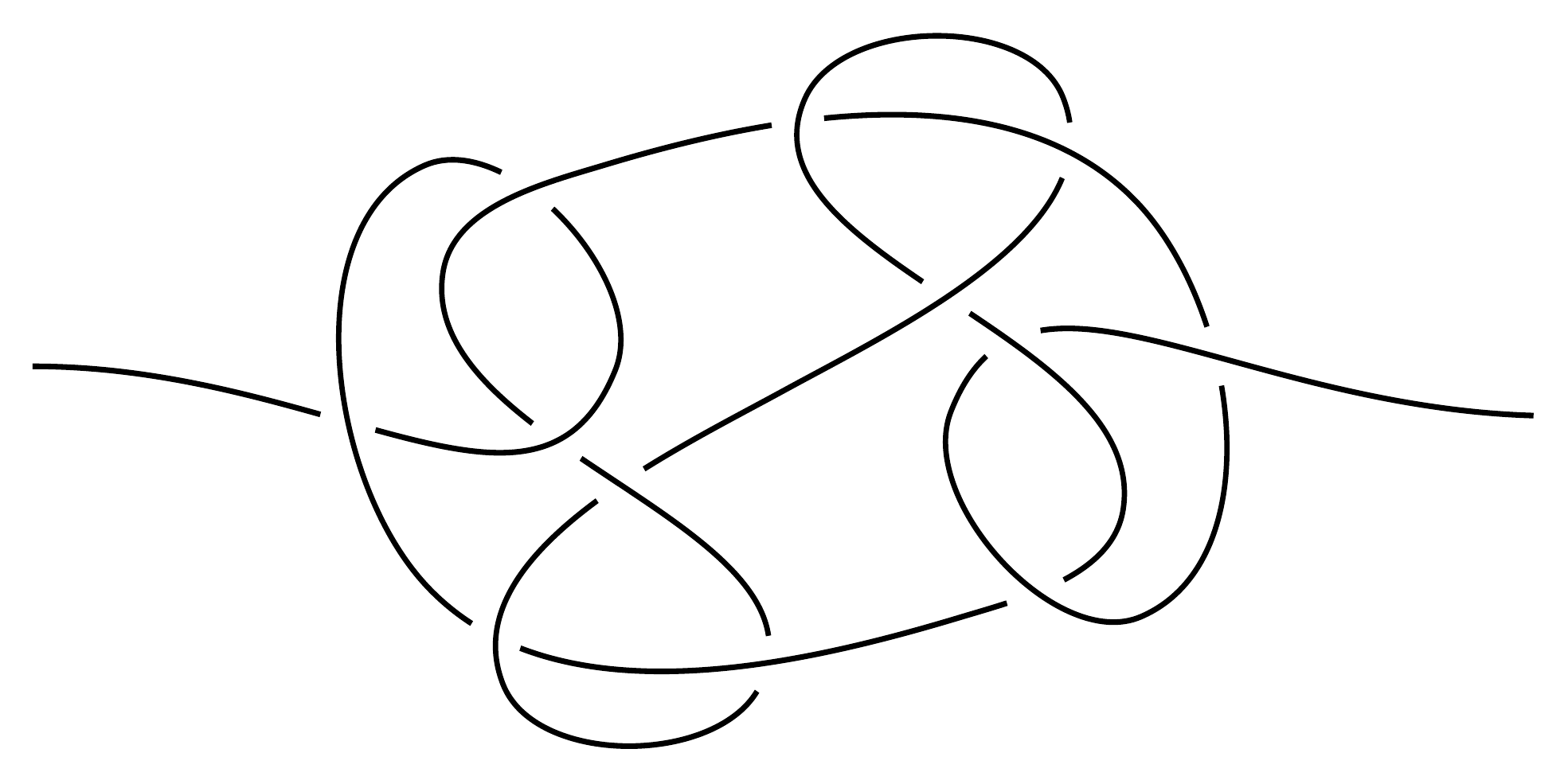}}}&&&\\*
 $12a_{435}$&&No&(Det)&$15^2$\\*
 &&&\\*
 &&&\\*
 &&&\\*
 &&&\\*
 &&&\\*
 &&&\\*
 &&&\\\hline\hline

 & \multirow{9}{*}{\scalebox{.4}{\includegraphics{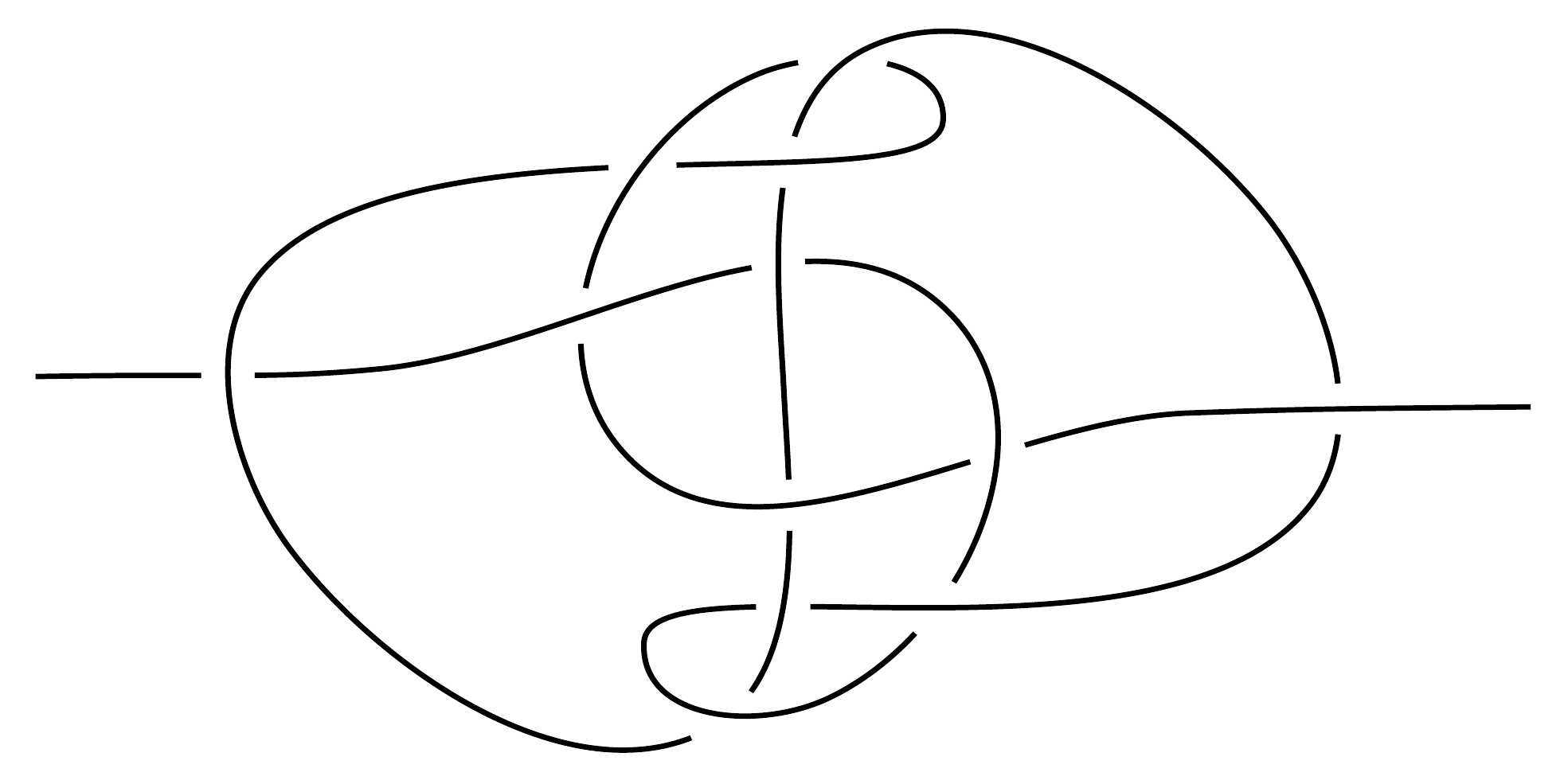}}}&&&\\*
 $12a_{458}$&&Unknown&(Unk)&$17^2$\\*
 &&&\\*
 &&&\\*
 &&&\\*
 &&&\\*
 &&&\\*
 &&&\\*
 &&&\\\hline\hline

 & \multirow{9}{*}{\scalebox{.4}{\includegraphics{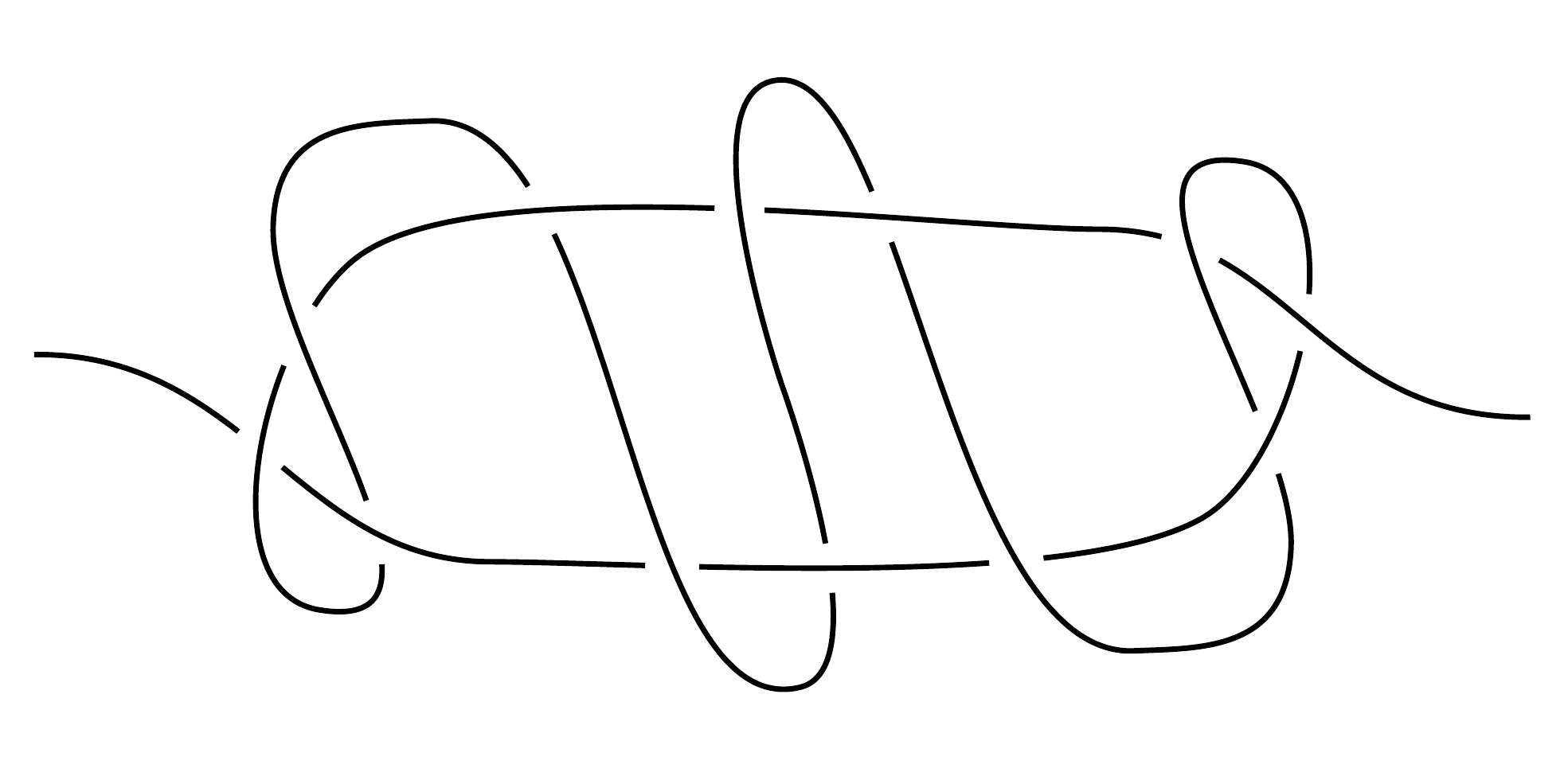}}}&&&\\*
 $12a_{477}$&&Unknown&(Unk)&$13^2$\\*
 &&&\\*
 &&&\\*
 &&&\\*
 &&&\\*
 &&&\\*
 &&&\\*
 &&&\\\hline\hline

 & \multirow{9}{*}{\scalebox{.4}{\includegraphics{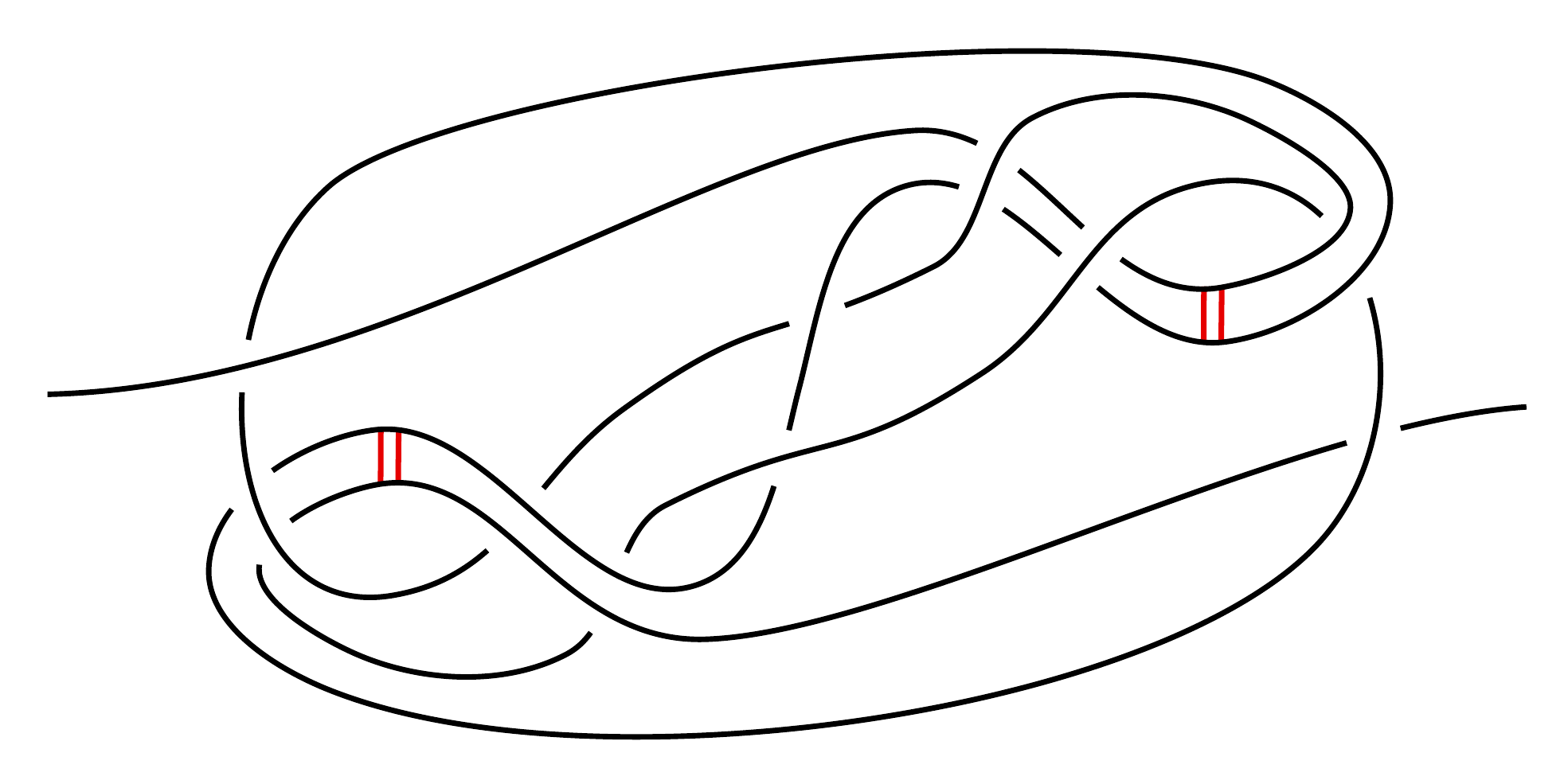}}}&&&\\*
 $12a_{819}$&&Yes&(Rib)&$13^2$\\*
 &&&\\*
 &&&\\*
 &&&\\*
 &&&\\*
 &&&\\*
 &&&\\*
 &&&\\\hline\hline

 & \multirow{9}{*}{\scalebox{.4}{\includegraphics{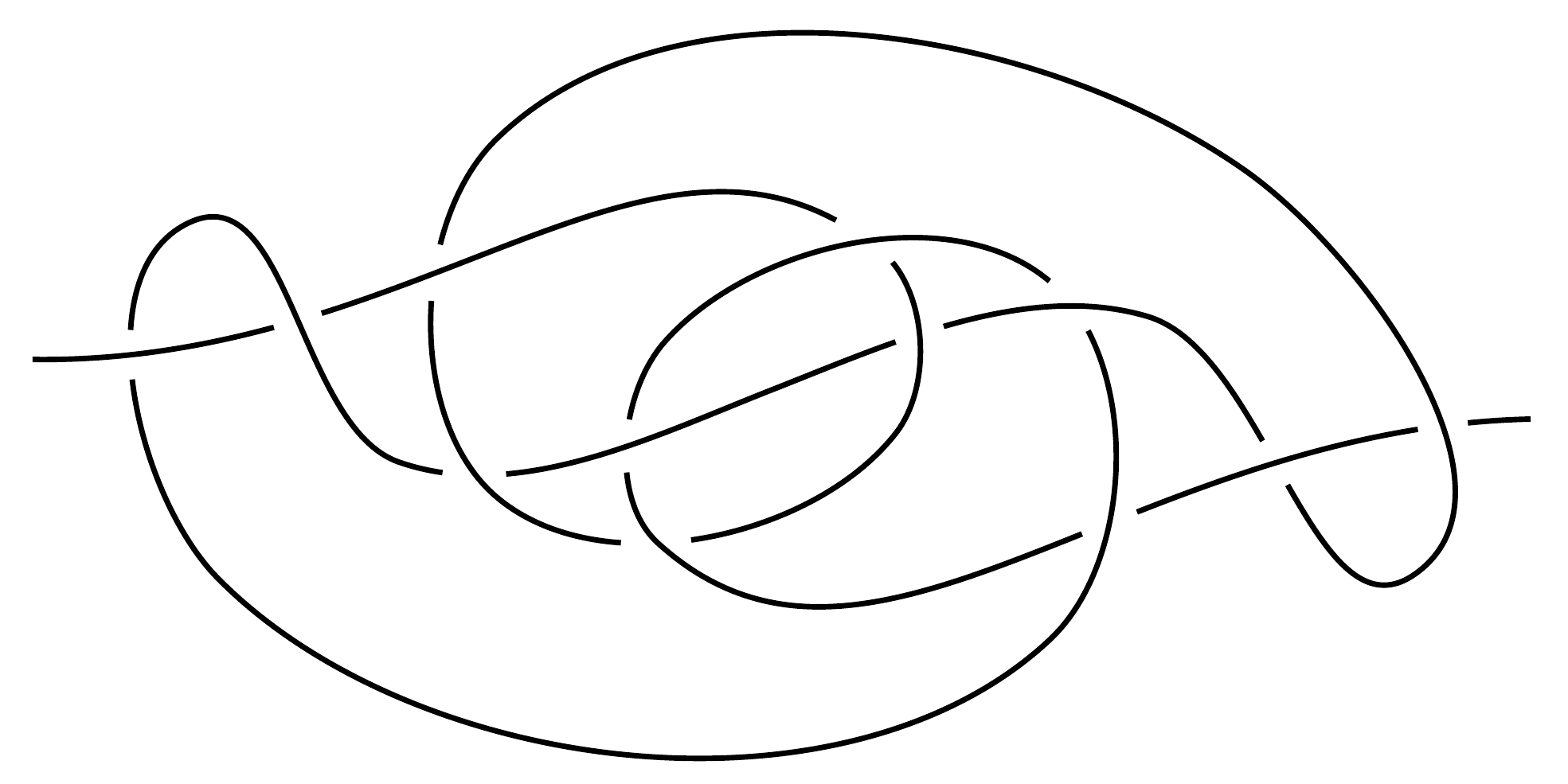}}}&&&\\*
 $12a_{887}$&&Unknown&(Unk)&$17^2$\\*
 &&&\\*
 &&&\\*
 &&&\\*
 &&&\\*
 &&&\\*
 &&&\\*
 &&&\\\hline\hline

 & \multirow{9}{*}{\scalebox{.4}{\includegraphics{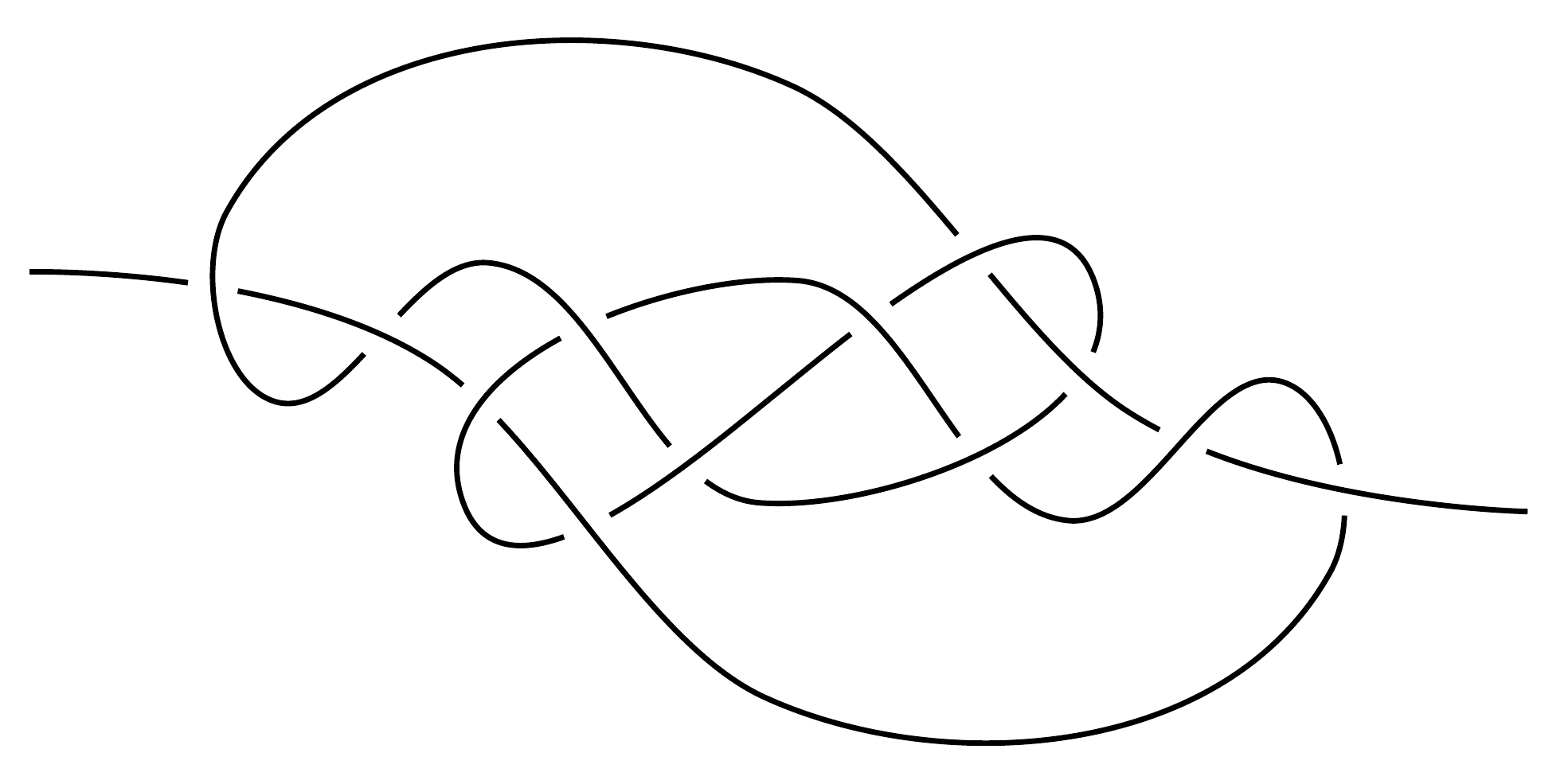}}}&&&\\*
 $12a_{990}$&&No&(Det)&$15^2$\\*
 &&&\\*
 &&&\\*
 &&&\\*
 &&&\\*
 &&&\\*
 &&&\\*
 &&&\\\hline\hline

 & \multirow{9}{*}{\scalebox{.4}{\includegraphics{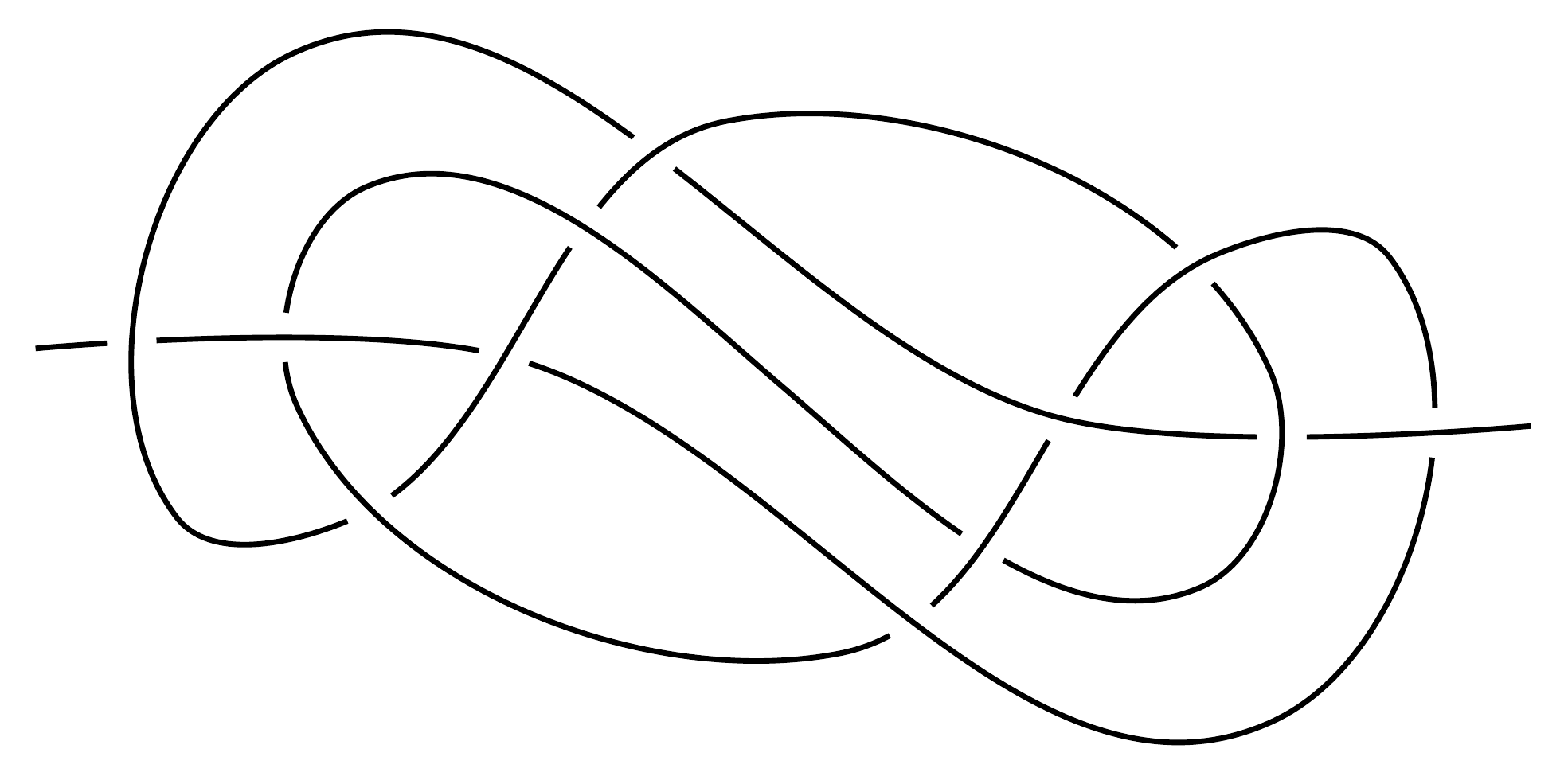}}}&&&\\*
 $12a_{1019}$&&No&(Det)&$19^2$\\*
 &&&\\*
 &&&\\*
 &&&\\*
 &&&\\*
 &&&\\*
 &&&\\*
 &&&\\\hline\hline

 & \multirow{9}{*}{\scalebox{.4}{\includegraphics{Inkscape Diagrams/12a1105.pdf}}}&&&\\*
 $12a_{1105}$&&No&($\spinc$)&$17^2$\\*
 &&&\\*
 &&&\\*
 &&&\\*
 &&&\\*
 &&&\\*
 &&&\\*
 &&&\\\hline\hline

 & \multirow{9}{*}{\scalebox{.4}{\includegraphics{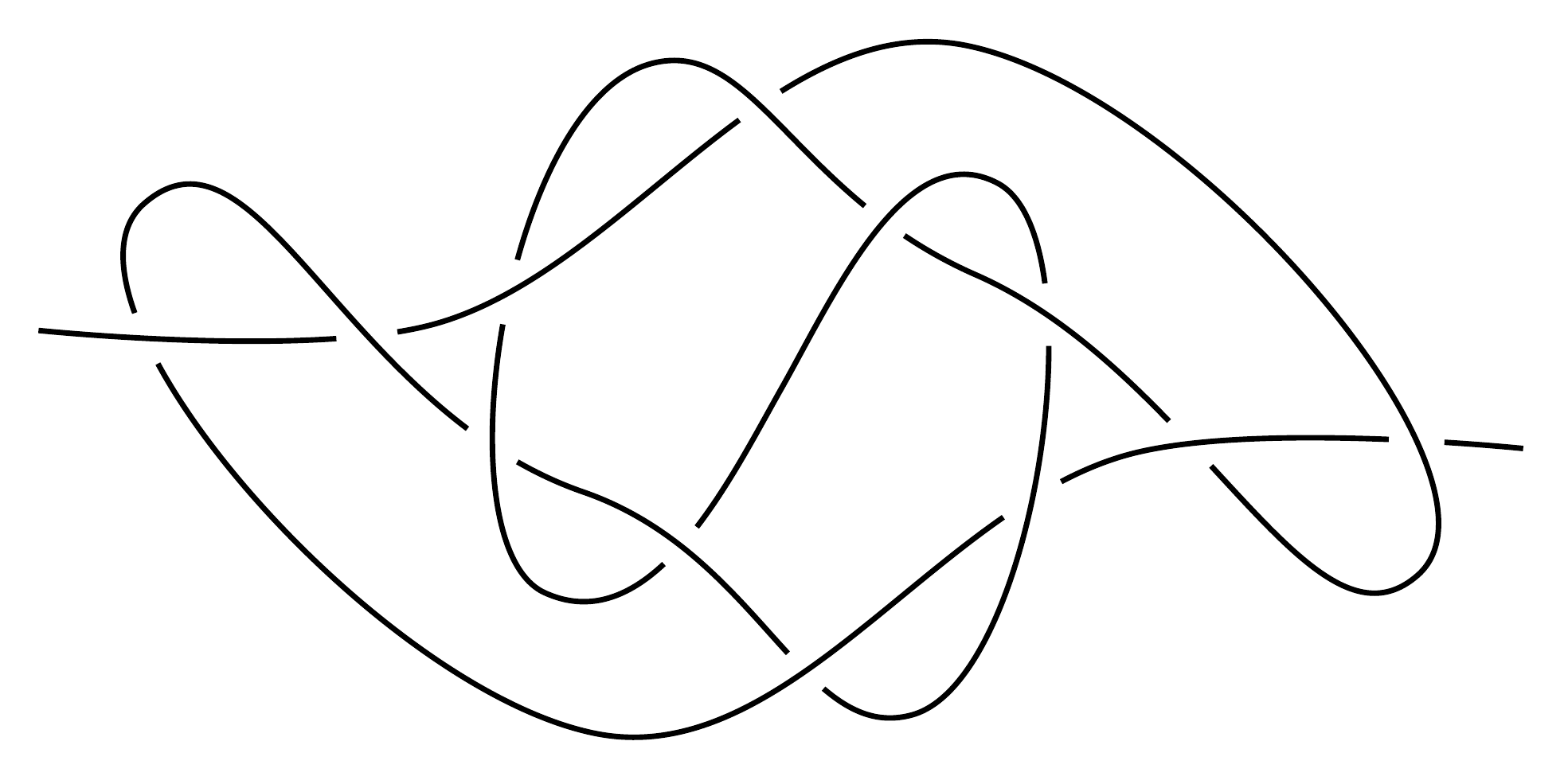}}}&&&\\*
 $12a_{1202}$&&No&($\spinc$)&$13^2$\\*
 &&&\\*
 &&&\\*
 &&&\\*
 &&&\\*
 &&&\\*
 &&&\\*
 &&&\\\hline\hline

 & \multirow{9}{*}{\scalebox{.4}{\includegraphics{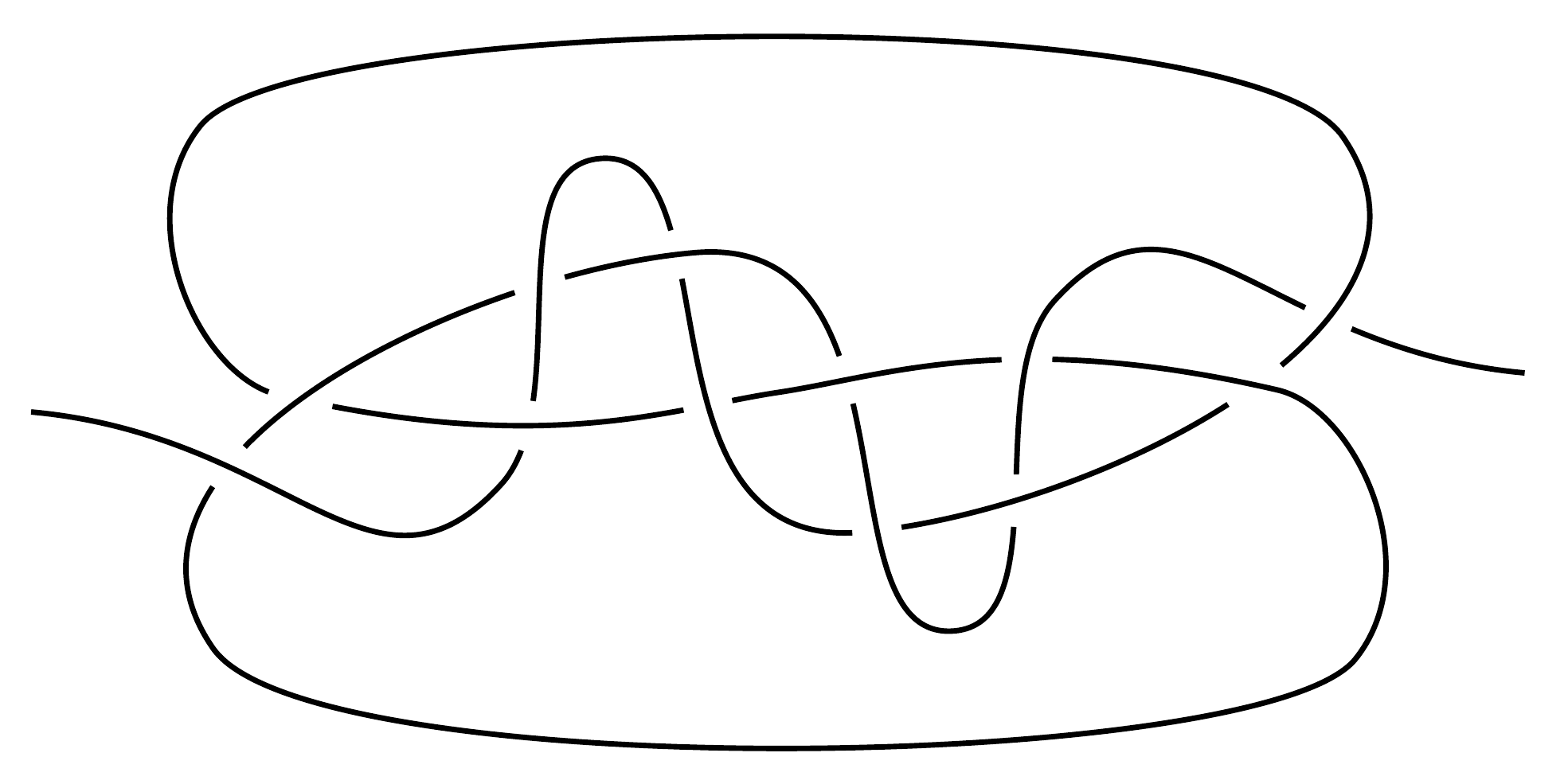}}}&&&\\*
 $12a_{1225}$&&No&(Det)&$15^2$\\*
 &&&\\*
 &&&\\*
 &&&\\*
 &&&\\*
 &&&\\*
 &&&\\*
 &&&\\\hline\hline

 & \multirow{9}{*}{\scalebox{.4}{\includegraphics{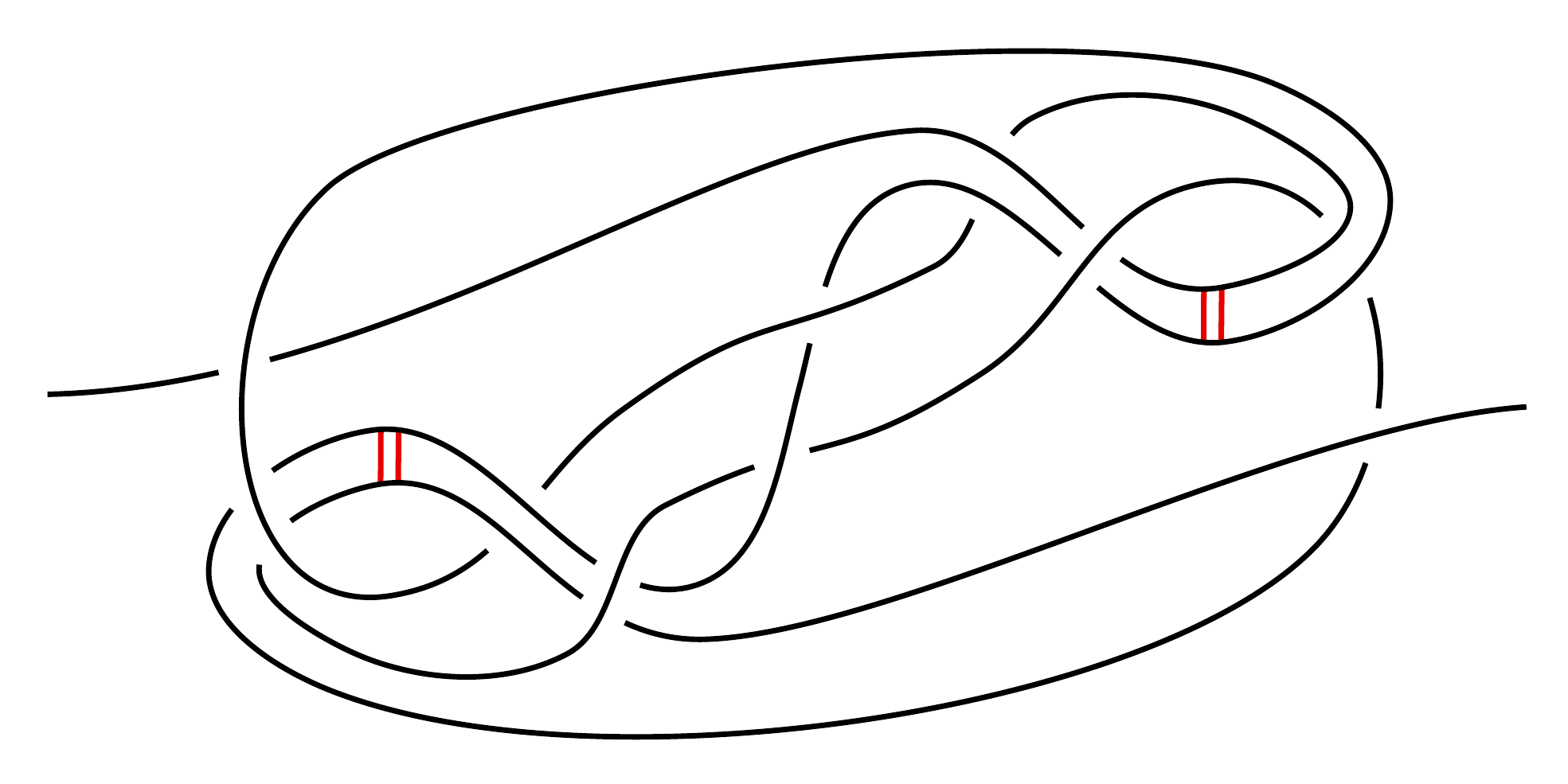}}}&&&\\*
 $12a_{1269}$&&Yes&(Rib)&$13^2$\\*
 &&&\\*
 &&&\\*
 &&&\\*
 &&&\\*
 &&&\\*
 &&&\\*
 &&&\\\hline\hline

 & \multirow{9}{*}{\scalebox{.4}{\includegraphics{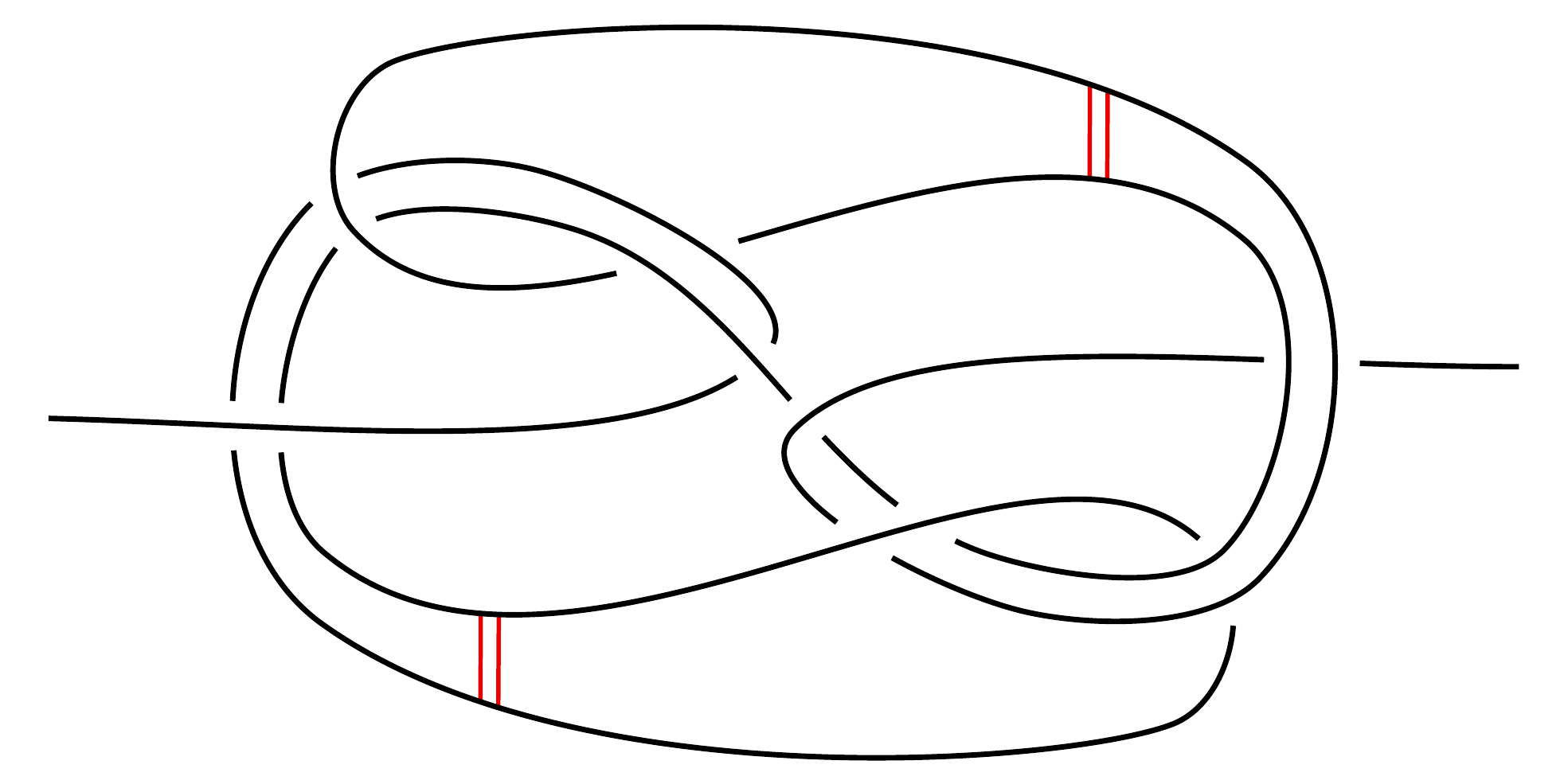}}}&&&\\*
 $12n_{462}$&&Yes&(Rib)&$5^2$\\*
 &&&\\*
 &&&\\*
 &&&\\*
 &&&\\*
 &&&\\*
 &&&\\*
 &&&\\\hline\hline

 & \multirow{9}{*}{\scalebox{.4}{\includegraphics{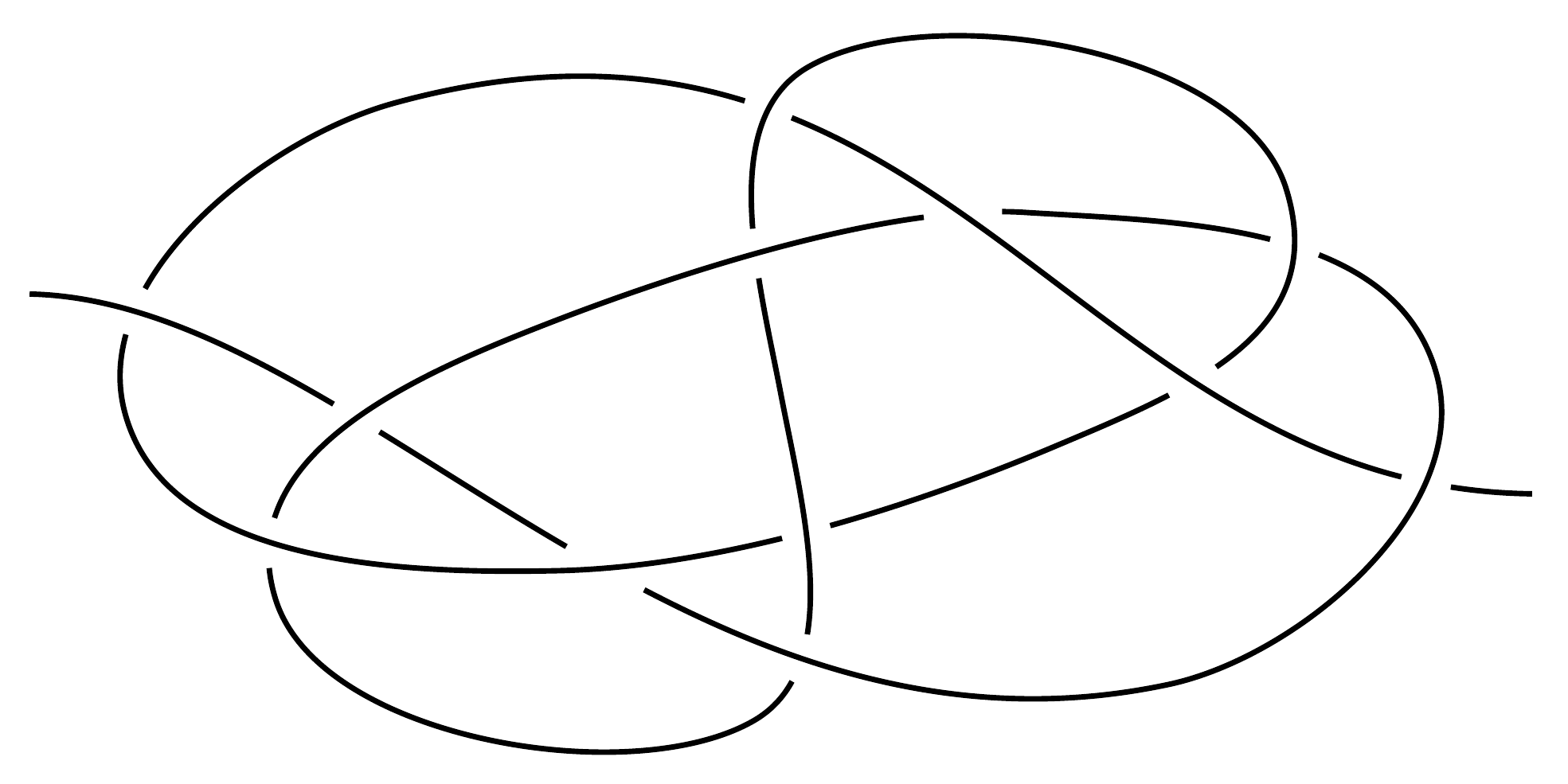}}}&&&\\*
 $12n_{706}$&&No&(Det)&$7^2$\\*
 &&&\\*
 &&&\\*
 &&&\\*
 &&&\\*
 &&&\\*
 &&&\\*
 &&&\\\hline\hline
\end{longtable}

\bibliography{bibliography}
\bibliographystyle{alpha}

\end{document}